\documentclass[11pt,a4 paper,twoside, reqno]{amsart}
\usepackage[utf8]{inputenc}
\usepackage[OT1]{fontenc}
\usepackage[english]{babel}
\usepackage{dsfont}
\usepackage{amsfonts}
\usepackage{amsmath}
\usepackage{bbm}
\usepackage{wasysym}
\usepackage{amsthm}
\usepackage{amssymb}
\usepackage{float}
\usepackage{textcomp}
\usepackage{xcolor}
\usepackage{graphicx}
\usepackage{fancybox}
\usepackage{fancyhdr}
\usepackage{mathrsfs}
\usepackage{tikz}
\usetikzlibrary{arrows}
\usetikzlibrary{calc}
\usepackage{centernot}
\usepackage{listings}
\usepackage{faktor}
\usepackage{bm}
\usepackage{setspace}
\usepackage{hyperref}
\usepackage{newlfont}
\usepackage{geometry}
\usepackage{ccaption}
\usepackage{tipa}
\usepackage{units}
\usepackage[autostyle,italian=guillemets]{csquotes}
\usepackage{comment}

\usepackage{guit}
\usepackage{tikz-cd}
\usepackage{enumerate}

\everymath{\displaystyle}

\theoremstyle{definition}
\newtheorem{Def}{Definition}[section]

\newtheorem{es}[Def]{Example}

\newtheorem*{def:flat}{Definition~\ref{flat}}
\theoremstyle{remark}
\newtheorem{obs}[Def]{Remark}
\theoremstyle{plain}
\newtheorem{prop}[Def]{Proposition}

\newtheorem{lema}[Def]{Lemma}

\newtheorem{cor}[Def]{Corollary}

\newtheorem{teo}[Def]{Theorem}

\newtheorem*{teo:main}{Theorem~\ref{main}}
\newcommand{\bo}{\mathbf}
\newcommand{\A}{{\mathcal A}}
\newcommand{\B}{{\mathcal B}}
\newcommand{\C}{{\mathcal C}}
\newcommand{\D}{{\mathcal D}}
\newcommand{\E}{{\mathcal E}}
\newcommand{\F}{{\mathcal F}}
\newcommand{\G}{{\mathcal G}}
\renewcommand{\H}{{\mathcal H}}

\newcommand{\K}{{\mathcal K}}
\renewcommand{\L}{{\mathcal L}}
\newcommand{\M}{{\mathcal M}}

\renewcommand{\P}{{\mathcal P}}

\renewcommand{\S}{{\mathcal S}}
\newcommand{\T}{{\mathcal T}}
\newcommand{\U}{{\mathcal U}}
\newcommand{\V}{{\mathcal V}}

\newcommand{\mt}{\mathfrak}
\newcommand{\tx}{\textnormal}

\newcommand{\op}{^\textnormal{op}}

\newcommand{\colim}{\operatornamewithlimits{colim}}

\makeatletter
\newcommand{\changeoperator}[1]{%
	\csletcs{#1@saved}{#1@}%
	\csdef{#1@}{\changed@operator{#1}}%
}
\newcommand{\changed@operator}[1]{%
	\mathop{%
		\mathchoice{\textstyle\csuse{#1@saved}}
		{\csuse{#1@saved}}
		{\csuse{#1@saved}}
		{\csuse{#1@saved}}%
	}%
}
\makeatother

\makeatletter
\def\@tocline#1#2#3#4#5#6#7{\relax
	\ifnum #1>\c@tocdepth % then omit
	\else
	\par \addpenalty\@secpenalty\addvspace{#2}%
	\begingroup \hyphenpenalty\@M
	\@ifempty{#4}{%
		\@tempdima\csname r@tocindent\number#1\endcsname\relax
	}{%
		\@tempdima#4\relax
	}%
	\parindent\z@ \leftskip#3\relax \advance\leftskip\@tempdima\relax
	\rightskip\@pnumwidth plus4em \parfillskip-\@pnumwidth
	#5\leavevmode\hskip-\@tempdima
	\ifcase #1
	\or\or \hskip 1em \or \hskip 2em \else \hskip 3em \fi%
	#6\nobreak\relax
	\hfill\hbox to\@pnumwidth{\@tocpagenum{#7}}\par% <---- \dotfill -> \hfill
	\nobreak
	\endgroup
	\fi}
\makeatother

\changeoperator{sum}
\changeoperator{prod}
\changeoperator{coprod}

\title{Flatness, weakly lex colimits, and free exact completions}
\author{Giacomo Tendas}
\address{Department of Mathematics and Statistics, Masaryk University, Faculty of Sciences, Kotl\'{a}\v{r}sk\'{a} 2, 611 37 Brno, Czech Republic}
\email{tendasg@math.muni.cz}
\date{\today}
\thanks{The author acknowledges with gratitude the useful comments and suggestions of the anonymous referee, as well as the support of the Grant agency of the Czech Republic under the grant 22-02964S}

\begin{document}
	
\begin{abstract}
	We capture in the context of {\em lex colimits}, introduced by Garner and Lack, the universal property of the free regular and Barr-exact completions of a weakly lex category. This is done by introducing a notion of {\em flatness} for functors $F\colon\C\to\E$ with lex codomain, and using this to describe the universal property of free $\Phi$-exact completions in the absence of finite limits, for any given class $\Phi$ of lex weights. In particular, we shall give necessary and sufficient conditions for the existence of free lextensive and free pretopos completions in the non-lex world, and prove that the ultraproducts, in the categories of models of such completions, satisfy an universal property.
\end{abstract}	
	
\maketitle
\begin{comment}
	{\small
		\noindent{\bf Keywords:} Flatness, regular/exact categories, free completions, lex colimits, enriched categories.\\
		{\bf Mathematics Subject Classification:} 18E08, 18A35, 18D20, 18B25.\\
		{\bf Competing Interests:} the author declares none.
	}
\end{comment}

\setcounter{tocdepth}{1}
\tableofcontents

\section{Introduction}

Regular and Barr-exact categories were first considered in \cite{barr1971exact} as the ordinary counterpart of abelian categories, and have found many applications in different areas of category theory, logic, and algebra. In particular, free regular and Barr-exact completions of lex categories (that is, categories with finite limits) have always played a key role \cite{CM82:articolo,Lac99:articolo,Ros99,van2005inductive}. Given a lex category $\C$, its free Barr-exact completion is determined by a Barr-exact category $\C_\tx{ex}$ together with an embedding $K\colon\C\hookrightarrow\C_\tx{ex}$ for which left Kan extending along $K$ induces an equivalence for any Barr-exact category $\E$
$$ \tx{Lex}(\C,\E)\simeq\tx{Reg}(\C_\tx{ex},\E)  $$
between finite-limit preserving (lex) functors $\C\to\E$ and regular functors $\C_\tx{ex}\to \E$; the latter being those lex functors which also preserves regular epimorphisms (or equivalently, coequalizers of kernel pairs). It was soon realized that such free completions could be considered even when the category $\mathcal C$ is not finitely complete; in that case, however, one needs to replace the notion of lex functor with some other concept. 

In fact, Carboni--Vitale~\cite{CV98:articolo} and Hu~\cite{Hu96:articolo} showed, independently, that free Barr-exact completions exist for any {\em weakly lex} category $\C$. A category is weakly lex if it has weak finite limits; a weak limit of a diagram $H\colon\D\to\C$ is determined by an object $\tx{wlim} H\in \C$ together with a cone $\delta\colon \Delta (\tx{wlim}H)\to H$ such that each cone over $H$ factors, not necessarily uniquely, through $\delta$. Since the factorization is not unique, weak limits are not generally identified up to isomorphism. Weakly lex categories, and their exact completions, turned out to be important in various fields, including that of accessible categories~\cite{Hu1992:book,Ten22:duality}, categorical logic and type theory~\cite{emmenegger2020exact, maietti2013elementary,birkedal1998type}, homological algebra~\cite{grandis2000weak}, triangulated categories~\cite{rosicky2001exact}, as well as categories of topological spaces~\cite{carboni2000locally}.

In this context, lex functors in the universal property are replaced by {\em left covering functors}; that is, those $F\colon\C\to \E$ such that for any weak limit $\tx{wlim}H$ of a finite diagram $H\colon\D\to\C$, the comparison map
$$ F(\tx{wlim}H)\twoheadrightarrow \lim FH $$
in $\E$ is a regular epimorphism. When $\C$ is actually lex, $F\colon\C\to \E$ is left covering if and only if it is lex \cite[Proposition~20]{CV98:articolo}; while, if $\C$ is just weakly lex, but $\E=\bo{Set}$, then $F\colon\C\to \bo{Set}$ is left covering if and only if it is flat (that is, its category of elements is filtered) \cite[Corollary~23]{CV98:articolo}. For general $\C$ and $\E$ we do not have a description in terms of preservation of actual limits or colimits. Moreover, the notion of left covering functor is the right choice only when considering free regular and Barr-exact completions; it is in fact unclear what concept should be used when considering, for instance, free lextensive or free pretopos completions of non-lex categories.

The main aim of this paper is to address this problem with the introduction of a notion of {\em flatness} for functors $F\colon\C\to\E$ from any small category $\C$ into any lex category $\E$. We do that in the framework of {\em lex colimits} introduced by Garner and Lack \cite{GL12:articolo}, which captures not only the notions of regular and Barr-exact categories, but many others including (infinitary) lextensive, coherent, and adhesive categories as well as pretopoi. 

Starting from a class of lex weights $\Phi$, which represents the type of colimits involved in the definition, Garner and Lack introduce a notion of $\Phi$-exact category and provide a formal way to present the {\em free $\Phi$-exact completion} $\Phi_l\C$ of a lex category $\C$. The category $\Phi_l\C$ can be obtained as an iterated closure of the representables in $[\C\op,\bo{Set}]$ under finite limits and, what they name, {\em $\Phi$-lex colimits}. We note now that such construction makes sense also when $\C$ is not lex. 

The universal property of the free $\Phi$-exact completions in this setting asserts that for any lex category $\C$ the category $\Phi_l\C$ is $\Phi$-exact and left Kan extending along the inclusion $K\colon\C\hookrightarrow\Phi_l\C$ induces an equivalence
$$ \tx{Lex}(\C,\E)\simeq\Phi\tx{-Ex}(\Phi_l\C,\E) $$
for any $\Phi$-exact category $\E$. The first step towards a generalization of free $\Phi$-exact completions for a non-lex $\C$, is then to understand how to replace lex functors in the universal property above. When $\E=\bo{Set}$, a functor $F\colon\C\to \bo{Set}$, for a lex $\C$, is lex if and only if it is {\em flat}. As mentioned above, also left covering functors, from a weakly lex $\C$ into $\bo{Set}$, are the same as flat functors. Therefore, the general notion we are looking for is something that, when restricted to $\bo{Set}$-valued presheaves, gives back flatness.

The concept of flatness originates in the theory of modules with the work of Serre~\cite{Serre1956}, and has since become an important tool first in the theory of additive categories~\cite{ObRo70:articolo}, and then in many other areas of category theory \cite{diaconescu1975change,BS1983,AR94:libro,Ten2022continuity}. There are several equivalent ways to express flatness. Traditionally one says that $F\colon\C\to \bo{Set}$ is flat if and only if its category of elements $\tx{El}(F)$ ---based on $\C\op$--- is {\em filtered}. This is well known to be equivalent to requiring that the left Kan extension $$\tx{Lan}_YF\colon[\C\op,\bo{Set}]\longrightarrow \bo{Set},$$ of $F$ along the Yoneda embedding, preserves finite limits of representable. Using this fact and the pointwise definition of Kan extension, we can generalize the concept to functors with codomain any lex category $\E$ in place of $\bo{Set}$ as follows.

\begin{def:flat}[$\V=\bo{Set}$]
	We say that a functor $F\colon\C\to\E$, into a lex category $\E$, is {\em flat} if and only if for any finite diagram $H\colon\D\to\C$, the colimit 
	$$ \colim \left(\text{\Large $\nicefrac{\C}{H}$}\xrightarrow{\ \pi\ }\C\xrightarrow{\ F\ }\E\right) $$
	exists in $\E$ and is isomorphic, through the comparison map, to $\lim FH$.
\end{def:flat} 

Here, {\Large $\nicefrac{\C}{H}$} is the category of cones $\Delta C\to H$ over $H$; this can be identified with the category of elements of the functor $\lim YH\colon\C\op\to\bo{Set}$. It follows that, when $\E=\bo{Set}$, the colimit above defines the value of $\tx{Lan}_YF$ at $\lim YH$, so that the isomorphism requested in the definition above is simply asking that 
$$ \tx{Lan}_YF(\lim YH)\cong \lim FH$$ 
and we recover the usual notion of flatness. Moreover, we shall see that a functor with lex domain is flat if and only if it is lex (Proposition~\ref{flatchar}).

Now it is possible to define free $\Phi$-exact completions of non-lex categories: we say that $K\colon\C\hookrightarrow \D$ exhibits $\D$ as the {\em free $\Phi$-exact completion of $\C$} if left Kan extending along $K$ induces an equivalence
$$ \tx{Flat}(\C,\E)\simeq \Phi\tx{-Ex}(\D,\E)$$
for any $\Phi$-exact category $\E$. Such completions need not exist in general. For instance, when $\Phi=\emptyset$ the free $\emptyset$-exact completion of $\C$ exists if and only if $\C$ is lex (Remark~\ref{notexists}). Our main theorem below will provide necessary and sufficient conditions for the free $\Phi$-exact completion to exist:

\begin{teo:main}[$\V=\bo{Set}$]
	The following are equivalent for a small category $\C$:\begin{enumerate}\setlength\itemsep{0.25em}
		\item the free $\Phi$-exact completion of $\C$ exists;
		\item $K\colon\C\hookrightarrow\Phi_l\C$ exhibits $\Phi_l\C$ as the free $\Phi$-exact completion of $\C$;
		\item $\Phi^\diamond[\C]$ has finite limits of diagrams landing in $\C$.
	\end{enumerate}
\end{teo:main}

By (2), it follows that the free $\Phi$-exact completion of $\C$, when it exists, coincides with the category $\Phi_l\C$ introduced before. The category $\Phi^\diamond[\C]$ in $(3)$ will be defined in Section~\ref{main-sect} as a particular full subcategory of $[\C\op,\bo{Set}]$ obtained by taking {\em wlex colimits}, a generalization of lex colimits to the flat setting. Finite limits in $\Phi^\diamond[\C]$ should be understood as {\em virtual finite limits} in $\C$ in the sense of \cite{LT22:limits}. 

When considering our motivating example given by the class of lex weights $\Phi_\tx{ex}$ for Barr-exact categories (or $\Phi_\tx{reg}$ for regular categories), the theorem above implies that the free Barr-exact completion of $\C$ exists if and only if $\C$ has weak finite limits; and that a functor $F\colon\C\to\E$, from a weakly lex $\C$ to a Barr-exact category $\E$, is flat if and only if it is left covering (see Section~\ref{weaklylex}). Thus we recover the results of \cite{CV98:articolo,Hu96:articolo} as part of the framework of lex colimits. %Various other applications of Theorem~\ref{main} to different notions of exactness will be discussed in Section~\ref{examples}.

\vspace{5pt}
\noindent {\bf Outline.} The main results of this paper will be carried out in the context of categories enriched over a locally finitely presentable and symmetric monoidal closed category.

In Section~\ref{flatness} we introduce the notion of flat $\V$-functor and prove some basic properties which will be used in Section~\ref{main-sect}, where we introduce free $\Phi$-exact completions in the flat world and prove the existence theorem (Theorem~\ref{main}). Examples are discussed in Section~\ref{examples}; beside free regular and Barr-exact completions of small weakly lex categories, we describe free infinitary lextensive completions of small categories with finite multilimits, and free pretopos completions of small categories with finite fc-limits \cite{beke2005flatness}.

In Section~\ref{universalultra} we prove that the ultraproducts involved in the context of free pretopos and lextensive completions satisfy an universal property, despite not being the usual categorical ultraproducts computed in the presence of products and filtered colimits. Finally, in Appendix~\ref{fc-orth-inj} we give a description of the categories of models of lextensive categories and pretopoi in terms of, respectively, finite-cone orthogonality and finite-cone injectivity conditions.

\section{Flatness}\label{flatness}

In this section we generalize the notion of flatness from $\V$-functors taking value in $\V$ to $\V$-functors $F\colon\C\to \E$ taking value in a (possibly large and) lex $\V$-category $\E$. For that, we fix a symmetric monoidal closed category $\V=(\V_0,\otimes,I)$ as our base of enrichment and follow \cite{Kel82:articolo} for standard notations, particularly involving weighted limits and colimits. As in \cite{GL12:articolo}, we further assume that $\V_0$ is locally finitely presentable as a closed category in the sense of Kelly \cite{Kel82:articolo}; meaning that $\V_0$ is locally finitely presentable and the finitely presentable objects are closed under tensor product and contain the unit. 

In this context we can talk about finite weighted limits \cite[Section~4]{Kel82:articolo}. A weight $M\colon\C\op\to\V$ is called {\em finite} if $\C$ has finitely many objects, and $\C(A,B)$ and $MA$ are finitely presentable in $\V_0$ for any $A,B\in\C$. Existence and preservation of finite weighted limits is equivalent to that of finite conical limits and powers by finitely presentable objects. In short we will say that a $\V$-category (resp. $\V$-functor) is {\em lex} if it has (resp. preserves) finite weighted limits.

Then, following Kelly, we define $F\colon\C\to\V$ to be {\em flat} if $\C$ is small and the $\V$-functor 
$$F*-\colon[\C\op,\V]\to\V,$$
taking colimits weighted by $F$, is lex. In other words, if $F$-weighted colimits commute in the base $\V$ with finite limits. By \cite[Proposition~3.4]{LT22:virtual}, for $F$ to be flat, it is enough that~$F*-$ preserves just the finite limits of representables. 

Unpacking the definition, this means that $F\colon\C\to\V$ is flat if and only if for any finite weight $N\colon\D\to\V$ and any diagram $H\colon\D\to\C$, the comparison map below is an isomorphism.
$$ F*\{N,YH\} \xrightarrow{\ \cong\ } \{N,FH\}$$
With $Y\colon\C\hookrightarrow[\C\op,\V]$ we shall always denote the Yoneda embedding.

Now, to generalize the notion of flatness we need to replace $\V$, in the codomain of $F$, with a lex category $\E$. In this context however, colimits weighted by $F$ are not well defined. To avoid this problem, by \cite{Kel82:libro}, we can replace $F*\{N,YH\}$ with $\{N,YH\}*F$ in the isomorphism above, so that $\{N,YH\}\colon\C\op\to\V$ is now considered as a weight. Therefore we can extend the definition of flatness as follows:

\begin{Def}\label{flat}
	We say that a $\V$-functor $F\colon\C\to\E$, from a small $\V$-category $\C$ into a lex $\V$-category $\E$, is {\em flat} if for any finite weight $N\colon\D\to\V$ and any diagram $H\colon\D\to\C$, the colimit $\{N,YH\}*F$ exists in $\E$ and the comparison map
	$$\{N,YH\}*F\xrightarrow{\ \cong\ }  \{N,FH\} $$
	is an isomorphism.
\end{Def}

It is clear from the observation above, that when $\E=\V$ we recover Kelly's notion of flatness. 

In the dual of \cite[Definition~4.15]{LT22:virtual}, the object $\{N,YH\}$ of $[\C\op,\V]$ is called the {\em virtual limit} of $H$ weighted by $N$. Under this notation, we can then interpret the colimit $\{N,YH\}*F$ as the image of such virtual limit through $F$. Thus, one can say that $F\colon\C\to\E$ is flat if $F$ sends finite virtual limits in $\C$ to actual limits in $\E$.

\begin{obs}
	When $\V=\bo{Set}$ and weighted finite limits are just the usual finite limits, a functor $F\colon\C\to\E$ is flat if and only if for any finite diagram $H\colon\D\to\C$, the colimit $\lim YH*F$ exists in $\E$ and the comparison map induces an isomorphism
	$$ (\lim YH)*F\cong   \lim FH, $$
	which can be rewritten as 
	$$ \colim \left(\text{\Large $\nicefrac{\C}{H}$}\xrightarrow{\ \pi\ }\C\xrightarrow{F}\E\right) \cong  \lim FH$$
	since, as mentioned in the introduction, the category {\Large $\nicefrac{\C}{H}$} of cones over $H$ can also be described as the category of elements $\tx{El}(\lim YH)$.
\end{obs}

\begin{obs}
	For a general $\V$, the simplification above is not possible, since we cannot reduce ourselves to conical colimits. However, using that the finite weighted limit $\{N,YH\}$ is computed pointwise in $\V$, for any $C\in\C$ we have an isomorphism
	$$ \{N,YH\}(C)\cong [\D,\V](N,\C(C,H-)), $$
	which says that the weight $\{N,YH\}$, determining the colimit, is given pointwise by taking the object of cylinders $N\Rightarrow\C(C,H-)$ in $\C$.
\end{obs}

\begin{obs}
	One should not confuse the notion introduced above with the following: we say that $F\colon\C\to\E$ is {\em representably flat} if $\E(E,F-)\colon\C\to\V$ is flat for any $E\in\E$; this notion was considered for instance in \cite[Section~6.3]{borceux1994handbook} where it is called just flatness. 
	
	\noindent Even when $\V=\E=\bo{Set}$, while our notion coincides with standard flatness, a functor $F\colon\C\to\bo{Set}$ is representably flat if and only if the power $F^J$ is flat for any set $J$. Thus, if $\C\op$ is not filtered (so that the terminal functor $\Delta 1\colon\C\to\bo{Set}$ is not flat) then there are no representably flat functors $F\colon\C\to\bo{Set}$ (since $F^\emptyset\cong \Delta 1$).
\end{obs}

\begin{obs}\label{kar}
	Other general notions of flatness where previously studied (in the ordinary context) by Karazeris~\cite{kar04} for functors $F\colon\C\to\E$ into a category $\E$ with a chosen Grothendieck topology $j$; this relies on $F$ satisfying three filteredness conditions in the internal logic of $(\E,j)$. Using this notion of $j$-flatness, which was later thoroughly studied by Shulman~\cite{shulman2012exact}, it is then possible to capture regular, exact, lextensive, and pretopos completions of non-lex categories --- by choosing in each case the relevant topology. 
	
	However, unlike the notions of \cite{kar04,shulman2012exact}, our flat functors do not depend on the choice of a Grothendieck topology on the codomain category, and thus provides the same notion for all the different examples mentioned above (see Sections~\ref{weaklylex}, \ref{fc}, and~\ref{lextesive-sect}). 
\end{obs}

\begin{obs}\label{large}
	As pointed out in the definition, our flat $\V$-functors $\C\to\E$ will always have a small domain $\C$; this is because the collection of all presheaves $F\colon \C\op\to\V$ forms the class of objects of a $\V$-category $[\C\op,\V]$ only when $\C$ is small (while that does not happen in general for a large $\C$). However, the definition, as well as all the results of this paper involving flatness, can be extended to any $\V$-category $\C$ arguing as below.\\
	Given any (possibly large) $\V$-category $\C$, we can consider an universe enlargement $\mathbb V$ of $\V$ (as in \cite[2.6]{Kel82:libro}) for which $\C$ is a small $\mathbb V$-category (meaning that $\mathbb V$ has limits and colimits of size greater than the cardinality of $\tx{Ob}(\C)$). This way the notion of flat $\mathbb V$-functor out of $\C$ is well defined and all the results below still apply since we can see every $\V$-category as a $\mathbb V$-category. 
\end{obs}

Given $F\colon\C\to\E$, with $\E$ lex, the left Kan extension $\tx{Lan}_YF$ need not exist; thus we cannot express flatness in the same way defined by Kelly for presheaves. To solve this we need to consider a codomain restriction of the Yoneda embedding to a suitable subcategory of $[\C\op,\V]$: 

\begin{Def}
	Given a small $\V$-category $\C$ we denote by $\C_l$ the full subcategory of $[\C\op,\V]$ spanned by the finite limits of representables, and by $Z\colon\C\hookrightarrow\C_l$ the inclusion.
\end{Def}

Since the Yoneda embedding preserves limits, $\C=\C_l$ whenever $\C$ is lex. We do not know if $\C_l$ is actually closed in $[\C\op,\V]$ under finite limits. 

Then the notion of flatness can be rephrased as follows:

\begin{prop}\label{flat-prop}
	A $\V$-functor $F\colon\C\to\E$, with $\C$ small and $\E$ lex, is flat if and only if $\tx{Lan}_ZF\colon\C_l\to\E$ exists and preserves finite limits of diagrams from $\C$. 
\end{prop}
\begin{proof}
	Note that to give a finite limit of the form $\{N,YH\}$ in $[\C\op,\V]$ is the same as giving an object of $\C_l$, and in that case $\tx{Lan}_ZF(\{N,YH\})$ is by definition the colimit $\{N,YH\}*F $. Thus the isomorphisms involved in the definition of flatness hold if and only if $\tx{Lan}_ZF$ exists preserves finite limits of objects from $\C$.
\end{proof}

\begin{obs}
	When $\E$ is cocomplete, so that the colimit $\{N,YH\}*F$ always exists, the condition above is saying that the $\V$-functor $$\tx{Lan}_YF\cong-*F\colon [\C\op,\V]\to\E$$ preserves finite limits of representables.
\end{obs}

A few properties that will be used in the next sections are:

\begin{prop}\label{flatchar}
	Let $\C$ be a small $\V$-category. Then\begin{enumerate}\setlength\itemsep{0.25em}
		\item every representable $\C(C,-)$ is flat;
		\item if $\C$ is lex, $F\colon\C\to\E$ is flat if and only if it is lex;
		\item if $\E$ is lex and $J\colon\C\hookrightarrow\E$ is fully faithful and dense, then $J$ is flat.
	\end{enumerate}
\end{prop}
\begin{proof}
	(1). If $C\in\C$ then $\tx{Lan}_Z\C(C,-)\cong\C_l(ZC,-)$, and $\C_l(ZC,-)$ preserves any existing limits; thus $\C(C,-)$ is flat.
	
	(2). If $\C$ is lex, then $\C=\C_l$ and $\tx{Lan}_ZF\cong F$. Thus $F$ is flat if and only it it is lex.
	
	(3). If $\E$ is lex and $J\colon\C\to\E$ is fully faithful and dense, then the Yoneda embedding $Y\colon\C\to[\C\op,\V]$ factors through $J$ and $\E$ is identified with a full subcategory of $[\C\op,\V]$ containing the representables. Because of this, and since $\E$ is lex, the inclusion of $\E$ in $[\C\op,\V]$ preserves finite limits; thus $\E$ contains $\C_l$ as well. It is now easy to see that $\tx{Lan}_ZJ$ is isomorphic to the inclusion of $\C_l$ into $\E$, which preserves any existing limit. It follows that $J$ is flat.
\end{proof}

We next recall from \cite{GL12:articolo} the notion of small-exact $\V$-category. Below we denote by $\P\D$ the free cocompletion of a category $\D$ under all colimits. When $\D$ is small then $\P\D=[\D\op,\V]$. Moreover $\P\D$ is lex whenever $\D$ is \cite[Remark~6.6]{DL07}.

\begin{Def}
	We say that a $\V$-category $\T$ is a {\em small-exact} if it is a left exact localization of $\P\D$ for some lex category $\D$. In other words, $\T$ is small-exact if there exists a fully faithful $J\colon\T\hookrightarrow\P\D$, into a lex $\D$, which has a lex left adjoint.
\end{Def}

By construction every small-exact $\V$-category is cocomplete. If $\D$ above can be chosen to be small, then $\T$ is called a $\V$-topos in \cite{GL12:articolo}.

The following can be seen as a generalization of \cite[Lemma~2.7]{LT21:articolo} and \cite[Proposition~3.4]{LT22:virtual} to our context.

\begin{prop}\label{V-topos}
	Let $\T$ be small-exact and $\C$ be a small $\V$-category. A $\V$-functor $F\colon\C\to\T$ is flat if and only if $$\tx{Lan}_YF\colon [\C\op,\V]\longrightarrow\T$$ is lex.
	If in addition we are given $J\colon\C\to\B$ into a small $\V$-category $\B$, the following hold: \begin{enumerate}\setlength\itemsep{0.25em}
		\item[(a)] if $  F  $ is flat then also $\tx{Lan}_J  F $ is;
		\item[(b)] if $J$ is fully faithful and $\tx{Lan}_{J} F  $ is flat then $  F  $ is flat as well.
	\end{enumerate}
\end{prop}
\begin{proof}
	Given a $\V$-functor $J\colon\C\to\B$ we can consider the following triangle
	\begin{center}
		\begin{tikzpicture}[baseline=(current  bounding  box.south), scale=2]
			
			\node (a) at (0.7,0.7) {$[\C\op,\V]$};
			\node (c) at (0, 0) {$[\B\op,\V]$};
			\node (d) at (1.8, 0) {$\T$};
			
			\path[font=\scriptsize]
			
			(c) edge [->] node [above] {$[J\op,\V]\ \ \ \ \ \ \ \ \ \ \ $} (a)
			(a) edge [->] node [above] {$\ \ \ \ \ \ \ \tx{Lan}_YF$} (d)
			(c) edge [->] node [below] {$\tx{Lan}_{Y'}(\tx{Lan}_JF)$} (d);
			
		\end{tikzpicture}	
	\end{center} 
	where $Y'$ is the Yoneda embedding for $\B$. This commutes since $\tx{Lan}_YF\circ [J\op,\V]$ is cocontinuous and its restriction to $\B$ coincides with $\tx{Lan}_JF$. Moreover, if $J$ is fully faithful, the isomorphisms below hold.
	\begin{align}\label{isos-lan}
			\tx{Lan}_YF&\cong \tx{Lan}_YF\circ \tx{id}_{[\B,\V]}\nonumber \\
			&\cong \tx{Lan}_YF\circ [J\op,\V]\circ \tx{Ran}_{J\op}\nonumber\\
			&\cong (\tx{Lan}_{Y'}(\tx{Lan}_JF))\circ \tx{Ran}_{J\op}
	\end{align}
	It is clear that if $\tx{Lan}_YF$ preserves all finite limits then $F$ is flat. Conversely, assume that $F\colon\C\to\T$ is flat, so that $\tx{Lan}_YF$ preserves finite limits of representables. Consider $\B:=\tx{Fin}^\dagger\C$, the free completion of $\C$ under finite limits, with $J\colon\C\to\B$ the inclusion. We want to prove that $\tx{Lan}_JF$ is lex (and hence flat). Note first that $\tx{Lan}_JF$ preserves finite limits of diagrams landing in $\C$: take a finite weight $N \colon \D\to\V$ and a diagram $H\colon \D\to \C$ then
	\begin{align}\label{fin-lim-reps}
		(\tx{Lan}_JF)\{ N ,JH\}:&=  \B(J-,\{ N ,JH\})*F\nonumber \\
		&\cong  \{ N \square,\C(-,H\square)\}*F-\nonumber \\
		&\cong \{ N \square, \C(-,H\square)*F- \}\\
		&\cong \{ N , F \circ H\}\nonumber \\
		&\cong \{ N , (\tx{Lan}_JF)\circ JH\}\nonumber
	\end{align}
	where (\ref{fin-lim-reps}) follows from the fact that $-* F\cong \tx{Lan}_YF$ preserves finite limits of representables.\\
	Now, every object of $\B$ is a $J$-absolute finite limit of a diagram landing in $\C$; therefore $\tx{Lan}_JF$ preserves all the $J$-absolute limits of a chosen codensity presentation of $J$. By \cite[Proposition~2.2]{Day77:articolo} we then have $\tx{Lan}_JF\cong \tx{Ran}_JF$. But $\B$ is the free completion of $\C$ under finite limits, therefore the functor $\tx{Ran}_JF$ is lex by the universal property of such completion. This means that $\tx{Lan}_JF$ is lex. To conclude then note that $\tx{Lan}_{Y'}(\tx{Lan}_JF)$ is also lex by \cite[Proposition~2.4(5)]{GL12:articolo}; thus $\tx{Lan}_YF$ is lex since it can be expressed as the composition of lex $\V$-functors by (\ref{isos-lan}). This proves the first part of the statement.
	
	The proof of $(a)$ and $(b)$ now follows easily from what just proved and the initial arguments of the proof. If $F$ is flat then $\tx{Lan}_YF$ is lex; thus also $\tx{Lan}_{Y'}(\tx{Lan}_JF)$ is lex (by the commutativity of the triangle above), and hence $\tx{Lan}_JF$ is flat. Conversely, if $J$ is fully faithful and $\tx{Lan}_JF$ is flat, then $\tx{Lan}_YF$ is lex being the composite of two lex functors by (\ref{isos-lan}), so that $F$ is flat too.
\end{proof}

\section{Free $\Phi$-exact completions}\label{main-sect}

Consider now a class $\Phi$ of {\em lex-weights} as in \cite{GL12:articolo}; that is, a class $\Phi$ of weights of the form $M\colon\C\op\to\V$ with $\C$ a small lex $\V$-category. By a {\em $\Phi$-lex colimit} in a $\V$-category $\E$ we mean a colimit of the form $M*H$ where $M\colon\C\op\to\V$ is in $\Phi$ and $H\colon\C\to\E$ is lex.

For each lex $\V$-category $\E$, we define the $\V$-category $\Phi_l\E$ as the closure of the representables in $\P\E$ under finite limits and $\Phi$-lex-colimits. Then \cite{GL12:articolo} (but not us) defines a $\V$-category $\E$ to be $\Phi$-exact if the inclusion $\E\hookrightarrow\Phi_l\E$ has a lex left adjoint. Because of the nature of our proofs, the definition of $\Phi$-exact $\V$-category that we consider is slightly stronger than the one just mentioned; see below and Remark~\ref{small}.

\begin{Def}\label{exact}
	We say that a $\V$-category $\E$ is $\Phi$-exact if it can be identified with a full subcategory of a small-exact $\V$-category $\T$ which is closed under finite limits and $\Phi$-lex colimits. A $\Phi$-exact $\V$-functor is one that preserves finite limits and $\Phi$-lex colimits.
\end{Def}

By \cite[Theorem~4.1]{GL12:articolo}, if a $\V$-category $\E$ is $\Phi$-exact in the sense defined above, then is also $\Phi$-exact in the sense of \cite{GL12:articolo}, and the two notions coincide whenever $\E$ is small. 

\begin{obs}\label{small}
	It is worth saying a few words on why we need to consider the notion of $\Phi$-exactness given above rather than the one of Garner and Lack. 
	
	\noindent In the proof of the main theorem of this section (Theorem~\ref{main}) we need to show that the left Kan extension of a flat $F\colon\C\to\E$, into a $\Phi$-exact $\E$, along the inclusion $\C\hookrightarrow\Phi_l\C$ is a lex $\V$-functor. The way we achieve this, is by proving it first for any small-exact $\V$-category $\T$ (using Proposition~\ref{V-topos}), and then by extending it to any $\E$ using an embedding as in the definition above. This makes such an embedding a key ingredient of our proof.
	
	We should point out however, that the embedding of Definition~\ref{exact} consists merely in a {\em smallness} assumption on our $\Phi$-exact $\V$-categories. Indeed, the $\Phi$-exact embedding into a small-exact $\T$ can be obtained also for any $\Phi$-exact $\V$-category $\E$ in the sense of \cite{GL12:articolo} if we allow $\T$ to be a small-exact $\mathbb V$-category for an appropriate universe enlargement $\mathbb V$ of $\V$. This would in fact be enough for our results. 
	
	\noindent The universe enlargement can be chosen as follows. Since $\V_0\simeq\tx{Lex}(\V\op_{0f},\bo{Set})$, where $\V_{0f}$ is the full subcategory of $\V_0$ spanned by its finitely presentable objects, then one can define $\mathbb V_0:=\tx{Lex}(\V\op_{0f},\bo{SET})$ for a larger universe of sets $\bo{SET}$ for which $\E$ is $\bo{SET}$-small (the monoidal structure on $\mathbb V$ is induced by that of $\V$). Then $\mathbb V$-finite limits are the same as $\V$-finite limits (since $\V_f\simeq \mathbb V_f$ by construction), Proposition~\ref{V-topos} still applies in this larger setting, and the $\mathbb V$-small version of \cite[Theorem~4.1]{GL12:articolo} provides an embedding of $\E$ into a small-exact $\mathbb V$-category $\T$. 
\end{obs}

\begin{obs}
	In \cite{GL12:articolo} a class of lex weights $\Phi$ is called {\em small} if $\Phi_l\E$ is small for any small and lex $\E$. Without considering universe enlargements, we shall see that when $\Phi$ is small (as it is in many of our examples), the main theorem of this section applies not only to our notion of $\Phi$-exact $\V$-category, but also to that of \cite{GL12:articolo}. See Theorem~\ref{main-small}.
\end{obs}

For any small and lex $\V$-category $\C$ the $\V$-category $\Phi_l\C$ is $\Phi$-exact and turns out to be the free $\Phi$-exact completion of $\C$ in the sense of \cite[Proposition~3.8]{GL12:articolo}.

We want to generalize this free completion to the setting where $\C$ is not lex, but satisfies some {\em weaker} finite completeness conditions. First note that the definition of $\Phi_l\C$, as the closure of representable in $[\C\op,\V]$ under finite limits and $\Phi$-lex colimits, makes sense for any small $\V$-category $\C$; and in that case $\Phi_l\C$ is still $\Phi$-exact by construction.

\begin{Def}
	Let $M\colon\C\op\to \V$ be a weight. We say that a lex $\V$-category $\E$ has {\em $M$-wlex} colimits if it has all weighted colimits of the form $M*H$ where $H\colon\C\to\E$ is flat.
\end{Def}

\begin{obs}
	Following the notation of \cite{GL12:articolo}, it would have been natural to call these {\em $M$-flat} colimits, rather than $M$-wlex. However, that notation is already used in the literature to denote those colimits that commute with $M$-weighted limits in $\V$; for instance, flat presheaves coincide with the weights which are $M$-flat with respect to any finite weight $M$. This would have resulted in a confusing notation; therefore we opted for something that recalls our principal application to the case of {\em weakly lex} categories (Section~\ref{weaklylex}).
\end{obs}

Recall that the {\em saturation} $\Phi^*$ of $\Phi$, in the lex setting, is defined by stating that a weight $M\colon\C\op\to\V$ lies in $\Phi^*$ if and only if $\C$ is lex and $M\in\Phi_l\C$. Equivalently, $M\in\Phi^*$ if and only if $\C$ is lex and $M$-lex colimits exists in any $\Phi$-exact $\V$-category and are preserved by any $\Phi$-exact $\V$-functor (this follows from \cite[3.4 and 3.5]{GL12:articolo}). We shall now define a saturation of $\Phi$ in the flat-world.

\begin{Def}
	We say that a weight $M\colon\C\op\to\V$ belongs to $\Phi^\diamond$ if every $\Phi$-exact $\V$-category admits $M$-wlex colimits and every $\Phi$-exact $\V$-functor preserves them.
	Given a small $\V$-category $\C$, we denote by $\Phi^\diamond[\C]$ the full subcategory of $[\C\op,\V]$ spanned by the weights which lie in $\Phi^\diamond$.
\end{Def}

When we say that every $\Phi$-exact $\V$-functor $F\colon\E\to\F$ preserves $M$-wlex colimits, we ask just that, for any flat $H\colon\C\to\E$, the colimit $M*FH$ exists and the comparison map into $F(M*H)$ is an isomorphism. We do not ask the $\V$-functor $FH$ to be flat. However, we will see that under some hypotheses flatness is preserved under composition with $\Phi$-exact $\V$-functors (Lemma~\ref{flat-pres}).

If $\C$ is lex and $M\colon\C\op\to\V$ is a weight, then $M$-wlex colimits coincide with $M$-lex colimits since flat $\V$-functors out of $\C$ are the same as lex $\V$-functors; thus $M\in\Phi^*$ if and only if $M\in\Phi^\diamond$. It follows that we have an inclusion $\Phi^*\subseteq\Phi^\diamond$ given pointwise by the equality
$$ \Phi_l\C=\Phi^\diamond[\C] $$
whenever $\C$ is lex. In general, only one of the inclusions holds:

\begin{lema}
	For any small $\V$-category $\C$ the following inclusions hold
	$$\C\subseteq \Phi^\diamond[\C]\subseteq\Phi_l\C $$
	as full subcategories of $[\C\op,\V]$.
\end{lema}
\begin{proof}
	For any $H\colon\C\to\E$ we have $\C(-,C)*H\cong HC$; thus $\C(-,C)$-weighted colimits exist and are preserved by any $\V$-functor (that is, they are absolute). As a consequence $\C(-,C)\in\Phi^\diamond$ and  $\C\subseteq \Phi^\diamond[\C]$.
	
	Now note that the inclusion $K\colon\C\hookrightarrow\Phi_l\C$ is flat by Proposition~\ref{flatchar} since $K$ is dense. 
	Consider then $M\in\Phi^\diamond[\C]$; by definition the colimit $M*K$ exists in $\Phi_l\C$ and is preserved by the inclusion $V\colon\Phi_l\C\hookrightarrow[\C\op,\V]$ (being $\Phi$-exact). This means that 
	$$V(M*K)\cong M*VK\cong M*Y\cong M,$$ 
	where $Y\colon\C\to[\C\op,\V]$ is the Yoneda embedding. Thus the inclusion of $\Phi^\diamond[\C]$ into $[\C\op,\V]$ factors through $V$.
\end{proof}

\begin{lema}\label{flat-pres}
	Suppose that $\Phi^\diamond[\C]$ has finite limits of diagrams in $\C$. Then:\begin{enumerate}\setlength\itemsep{0.25em}
		\item $\tx{Lan}_ZH\colon\C_l\to\E$ exists for any $H\colon\C\to\E$ into a $\Phi$-exact $\E$;
		\item if $F\colon\E\to\F$ is a $\Phi$-exact $\V$-functor between $\Phi$-exact $\V$-categories and $H\colon\C\to\E$ is flat, then also $FH$ is flat.
	\end{enumerate}
\end{lema}
\begin{proof}
	The fact that $\Phi^\diamond[\C]$ has finite limits of diagrams in $\C$ implies that $\C_l\subseteq \Phi^\diamond[\C]$ as full subcategories of $[\C\op,\V]$. Moreover, using that the colimits $M\cong M*Z$, for $M\in\C_l$, give a density presentation for $Z\colon\C\hookrightarrow\C_l$, we know that the left Kan extension of a $\V$-functor $H\colon\C\to\E$ along $Z$ exists if and only if $M*H$ exists in $\E$ for any $M\in\C_l$, and in that case $\tx{Lan}_ZH(M)\cong M*H$.
	 
	Now, since $\C_l\subseteq \Phi^\diamond[\C]$, if $H\colon\C\to\E$ is flat the colimits $M*H$, for $M\in\C_l$, always exist in $\E$ and are preserved by any $\Phi$-exact $\V$-functor $F$. Then (1) follows at once by the argument above.
	
	For (2), given $H\colon\C\to\E$ flat, we know that $\tx{Lan}_ZH$ exists and preserves finite limits of objects from $\C$ (Proposition~\ref{flat-prop}). Moreover, since $F$ preserves the colimits $M*H$ for $M\in\C_l$, it follows that the composite $F\circ \tx{Lan}_ZH$ sends $M$ to the colimit $M*FH$. Thus
	$$ F\circ \tx{Lan}_ZH\cong\tx{Lan}_Z(FH).$$ 
	Since $F$ is lex, $\tx{Lan}_Z(FH)$ preserves finite limits of objects of $\C$; hence $FH$ is flat. 
\end{proof}

Next, we introduce the concept of free $\Phi$-exact completion in the flat setting. When $\C$ is lex, since flat and lex $\V$-functors out of $\C$ are equivalent notions, we recover the free $\Phi$-exact completion of \cite{GL12:articolo}.

\begin{Def}
	Let $K\colon\C\hookrightarrow \D$ be a fully faithful $\V$-functor between a small $\V$-category $\C$ and a $\Phi$-exact $\V$-category $\D$. We say that $K$ exhibits $\D$ as the {\em free $\Phi$-exact completion of $\C$} if left Kan extending along $K$ induces an equivalence
	$$ \tx{Flat}(\C,\E) \simeq\Phi\tx{-Ex}(\D,\E) $$
	for any $\Phi$-exact $\V$-category $\E$.
\end{Def}

When the free $\Phi$-exact completion exists, the inverse of the equivalence above is necessarily given by restricting along $K$. Note however that such a free $\Phi$-exact completion need not exist in general, as the following remark shows.

\begin{obs}\label{notexists}
	Let $\Phi=\emptyset$ and $\C$ be a small Cauchy complete $\V$-category that does not have finite limits; then the free $\emptyset$-exact $\V$-category over $\C$ does not exist. Indeed, by taking $\E=\V$ in the universal property, we would obtain an equivalence
	$$ \tx{Flat}(\C,\V) \simeq\tx{Lex}(\D,\V)$$
	implying that $\tx{Flat}(\C,\V)$ is locally finitely presentable. Since $\C$ is Cauchy complete, we would then obtain that $\C\simeq\D$ is lex, leading to a contradiction. Conversely, if $\C$ is lex then the free $\emptyset$-exact $\V$-category over $\C$ is $\C$ itself.
\end{obs}

If a free $\Phi$-exact completion of $\C$ exists, then it is unique up to equivalence and must coincide with $\Phi_l\C$:

\begin{lema}\label{unique-free}
	Let $K\colon\C\hookrightarrow \D$ exhibit $\D$ as the free $\Phi$-exact completion of $\C$, then there exists an equivalence $E\colon\D\to\Phi_l\C$ for which $E\circ K$ is isomorphic to the inclusion $\C\hookrightarrow\Phi_l\C$.
\end{lema}
\begin{proof}
	Since the identity $1_\D\colon\D\to\D$ is $\Phi$-exact, by the universal property we know that $K$ is flat and  $\tx{Lan}_K K\cong 1_\D$. Therefore $K$ is a dense $\V$-functor and hence we have an inclusion $W\colon\D\hookrightarrow[\C\op,\V]$ with $W\cong\tx{Lan}_KY$, where $Y\colon\C\to[\C\op,\V]$ is the Yoneda embedding. Again by the universal property, since $K$ is flat, it follows that $W$ is $\Phi$-exact. By definition of $\Phi_l\C$ it follows that $\Phi_l\C\subseteq \D$ as full subcategories of $[\C\op,\V]$; call $E\colon\Phi_l\C\to\D$ the inclusion. Now, the inclusion $K'\colon\C\hookrightarrow\Phi_l\C$ is flat; then $\tx{Lan}_KK'\colon\D\to\Phi_l\C$ exists and is $\Phi$-exact. By uniqueness, we know that $E\circ \tx{Lan}_KK'\cong 1_{\D}$, so that $E$ is essentially surjective on objects and hence an equivalence. This also satisfies $E\circ K\cong K'$ by construction.
\end{proof}

The following lemma will be important to characterize when free $\Phi$-exact completions exist.

\begin{lema}\label{diamond-lex}
	Let $\C$ be a small $\V$-category for which $\Phi^\diamond[\C]$ has finite limits of diagrams landing in $\C$. Then $\Phi^\diamond[\C]=\Phi_l\C$ is $\Phi$-exact.
\end{lema}
\begin{proof}
	It is enough to prove that $\Phi^\diamond[\C]$ is closed in $[\C\op,\V]$ under finite limits and $\Phi$-lex colimits; denote the inclusion by $W\colon  \Phi^\diamond[\C]\hookrightarrow[\C\op,\V]$. Then, since $\Phi_l\C$ is minimal among the full subcategories of $[\C\op,\V]$ satisfying this property and $ \Phi^\diamond[\C]\subseteq\Phi_l\C $, it follows that $\Phi^\diamond[\C]=\Phi_l\C$ is $\Phi$-exact.
	
	Let us first prove that $\Phi^\diamond[\C]$ is closed in $[\C\op,\V]$ under finite limits. Consider a finite weight $N\colon\A\to\V$ and a diagram $S\colon\A\to\Phi^\diamond[\C]$, it is enough to show that the limit $\{N,WS\}$ is in $\Phi^\diamond$. For that, consider a $\Phi$-exact $\V$-category $\E$ and a flat $\V$-functor $F\colon\C\to\E$. By definition there exists a fully faithful and $\Phi$-exact $J\colon\E\hookrightarrow\T$ into a small-exact $\T$; then $JF$ is flat by Lemma~\ref{flat-pres} and the colimit $\{N,WS\}*JF$ exists (by cocompleteness of $\T$) and is isomorphic to $\{N,WS*JF\}$ by Proposition~\ref{V-topos}. Here, the $\V$-functor $$WS*JF\colon\A\to\T$$ is given by sending $A\in\A$ to $WS(A)*JF$. Since $WS(A)\in\Phi^\diamond$ for any $A\in\A$, the $\V$-functor $WS*F\colon\A\to\E$ is well defined and $J(WS*F)\cong WS*JF$. It follows that $$\{N,WS*JF\}\cong \{N,J(WS*F)\}\cong J\{N,WS*F\}$$
	and therefore $\{N,WS\}*JF$, being isomorphic to the object above, actually lies in $\E$. Thus $\{N,WS\}*F$ exists in $\E$ and coincides with $\{N,WS*F\}$. This also shows that such colimit is preserved by any $\Phi$-exact $\V$-functor. Therefore $\{N,WS\}$ is in $\Phi^\diamond[\C]$.
	
	\noindent We now need to prove that $\Phi^\diamond[\C]$ is closed in $[\C\op,\V]$ under $\Phi$-lex colimits. Consider therefore $M\colon\D\op\to\V$ in $\Phi$ and a lex $\V$-functor $H\colon\D\to\Phi^\diamond[\C]$; as before, it is enough to prove that $M*WH$ lies in $\Phi^\diamond$. Consider hence a flat diagram $F\colon\C\to\E$ into a $\Phi$-exact $\E$; then the colimits $(M*WH)*F$ exists in $\E$ (and is preserved) if and only if $M*(WH*F)$ exists (and is preserved), and in that case they coincide. As before, the $\V$-functor $WH*F\colon\D\to\E$ is well defined since $WH(D)\in\Phi^\diamond$ for any $D\in\D$.  
	
	\noindent Thus, to conclude it is enough to show that $WH*F$ is lex; in fact then $(M*WH)*F$ can be seen as a $\Phi$-lex colimit. Consider first the case when $\E$ is small-exact; then for any finite weight $N\colon\A\to\V$ and $\S\colon\A\to\E$
	\begin{align}
		(WH*F)\{N,S\}&\cong WH(\{N,S\})*F\tag{by Def} \\
		&\cong \{N-,WHS-\}*F\tag{WH lex}\\
		&\cong \{N-,WHS-*F\}\tag{Prop~\ref{V-topos}}\\
		&\cong \{N,(WH*F)\circ S\}\nonumber
	\end{align}
	showing that $WH*F$ is lex. Now, given any $\Phi$-exact $\V$-category $\E$, there exists a fully faithful and $\Phi$-exact $J\colon\E\hookrightarrow\T$ into a small-exact $\T$. Since $\Phi^\diamond[\C]$ is lex, $JF$ is still flat by Lemma~\ref{flat-pres}; thus $WH*JF$ is lex by the previous arguments. But $WH*JF\cong J(WH*F)$ since any $\Phi$-exact $\V$-functor preserves $\Phi^\diamond$-flat colimits. Thus $WH*F$ is lex since $J$ is fully faithful.
\end{proof}

Below we give equivalent conditions for $\Phi_l\C$ to be the free $\Phi$-exact $\V$-category on $\C$.

\begin{teo}\label{main}
	The following are equivalent for a small $\V$-category $\C$:\begin{enumerate}\setlength\itemsep{0.25em}
		\item the free $\Phi$-exact completion of $\C$ exists;
		\item $K\colon\C\hookrightarrow\Phi_l\C$ exhibits $\Phi_l\C$ as the free $\Phi$-exact completion of $\C$;
		\item $\Phi^\diamond[\C]=\Phi_l\C$;
		\item $\Phi^\diamond[\C]$ has finite limits of diagrams landing in $\C$.
	\end{enumerate}
\end{teo}
\begin{proof}
	$(1)\Leftrightarrow(2)$ is given by Lemma~\ref{unique-free}, $(3)\Rightarrow(4)$ is trivial since $\Phi_l\C$ is lex by definition, and $(4)\Rightarrow(3)$ follows from Lemma~\ref{diamond-lex} above.
	
	$(2)\Rightarrow (3)$. Since $\Phi^\diamond[\C]\subseteq\Phi_l\C$ always holds, it s enough to prove that every $M\in\Phi_l\C$ lies in $\Phi^\diamond[\C]$. Fix an element $M\in\Phi_l\C$, then by density of $K$ we have that $M\cong M*K$ in $\Phi_l\C$, and such colimit is preserved by any $\V$-functor $\Phi_l\C\to\E$ which is the left Kan extension of its restriction to $\C$. \\
	To prove that $M\in\Phi^\diamond[\C]$, consider a flat $\V$-functor $H\colon\C\to\E$ into a $\Phi$-exact $\E$. By $(2)$ the left Kan extension $\tx{Lan}_KH$ exists and by the arguments above
	$$\tx{Lan}_KH(M)\cong \tx{Lan}_KH(M*K)\cong M*(\tx{Lan}_KH\circ K)\cong M*H $$
	so that $M*H$ exists in $\E$. Let now $F\colon \E\to\F$ be any $\Phi$-exact $\V$-functor into another $\Phi$-exact $\V$-category $\F$. Then 
	$$ F\circ\tx{Lan}_KH \cong \tx{Lan}_K(FH)$$
	by $(2)$ since the $\V$-functor on the left is $\Phi$-exact and extends $FH$. It follows that 
	$$ F(M*H)\cong F (\tx{Lan}_KH(M) )\cong\tx{Lan}_K(FH)(M)\cong M*FH $$
	and therefore $M\in \Phi^\diamond[\C]$.
	
	$(3)\Rightarrow (2)$. First notice that any $\V$-functor $G\colon\Phi_l\C\to\E$ into a $\Phi$-exact $\E$ is the left Kan extension of its restriction to $\C$ if and only if it preserves the colimit $M*K$ for any $M\in \Phi_l\C$. This follows from the fact that $K$ is dense and the $M*K$ form a density presentation of it. In particular, since every $M\in\Phi_l\C$ is by hypothesis in $\Phi^\diamond$ and $K$ is flat, every $\Phi$-exact $\V$-functor preserves said density presentation; thus $G\cong \tx{Lan}_K (GK)$ whenever $G\colon\Phi_l\C\to\E$ is $\Phi$-exact.\\
	Now we can show that acting by pre-composition with $K\colon\C\hookrightarrow\Phi_l\C$ induces a $\V$-functor 
	$$ \Phi\tx{-Ex}(\Phi_l\C,\E)\longrightarrow\tx{Flat}(\C,\E) $$ 
	for any $\Phi$-exact $\E$. For, let $W\colon\C_l\hookrightarrow\Phi_l\C$ be the inclusion, so that $W\circ Z\cong K$; then for any $\Phi$-exact $G\colon\Phi_l\C\to\E$ as above
	$$ G\circ W\cong \tx{Lan}_K (GK) \circ W\cong \tx{Lan}_W(\tx{Lan}_Z(GK))\circ W\cong \tx{Lan}_Z(GK); $$  
	therefore $\tx{Lan}_Z(GK)$ exists and preserves finite limits of elements from $\C$ (since both $G$ is lex and $W$ preserves them). So $GK$ is flat.\\
	To conclude we are only left to prove that for any flat $F\colon\C\to\E$ the left Kan extension $\tx{Lan}_KF$ exists and is $\Phi$-exact. By the arguments above, $\tx{Lan}_KF$ exists for any $F$ (as long as $\E$ is $\Phi$-exact). Thus we only need to prove that if $F$ is flat then $\tx{Lan}_KF$ is $\Phi$-exact.
	
	Consider first the case of a small-exact $\T$; then $\tx{Lan}_KF$ is the restriction of the left Kan extension $\tx{Lan}_YF$ along the Yoneda embedding $Y$. By Proposition~\ref{V-topos}, $\tx{Lan}_YF$ is lex and cocontinuous; therefore $\tx{Lan}_KF$ is lex and preserves all colimits that are computed pointwise in $[\C\op,\V]$. It follows that $\tx{Lan}_KF$ is $\Phi$-exact. Now, if $\E$ is any $\Phi$-exact $\V$-category, by definition there exists a fully faithful and $\Phi$-exact $H\colon\E\hookrightarrow\T$ into a small-exact $\T$. Since $H$ is $\Phi$-exact it preserves the colimits in $\E$ involved in the density presentation of $Z$, therefore 
	$$ H\circ\tx{Lan}_KF\cong \tx{Lan}_K(HF). $$
	By Lemma~\ref{flat-pres} the $\V$-functor $HF$ is still flat. Therefore $\tx{Lan}_K(HF)$ is $\Phi$-exact by the previous results, and hence also $\tx{Lan}_KF$ is since $H$ is fully faithful.
\end{proof}

In the case of a small $\Phi$ we obtain the following:

\begin{teo}\label{main-small}
	Let $\Phi$ be a small class of lex weights. The conditions of Theorem~\ref{main} above are further equivalent to:\begin{enumerate}
		\item[(5)]  left Kan extending along $K\colon\C\to\Phi_l\C$ induces an equivalence
		$$ \tx{Flat}(\C,\E) \simeq\Phi\tx{-Ex}(\Phi_l\C,\E) $$
		for any $\E$ that is a lex localization of $\Phi_l\E$ (that is, for any $\E$ that is $\Phi$-exact in the sense of \cite{GL12:articolo}).
	\end{enumerate}
\end{teo}
\begin{proof}
	It is clear that $(5)$ implies $(2)$ of Theorem~\ref{main} since every $\Phi$-exact $\V$-category (in our sense) is a lex localization of $\Phi_l\E$ (\cite[Theorem~4.1]{GL12:articolo}). Conversely, let us assume that $\Phi$ is small and that $(2)$ holds. Consider a $\V$-category $\E$ for which the inclusion $\E\hookrightarrow\Phi_l\E$ has a lex left adjoint. Arguing as in the beginning of $(3)\Rightarrow(2)$ of the proof above, we see that precomposition by $K$ induces a $\V$-functor
	$$ \Phi\tx{-Ex}(\Phi_l\C,\E)\longrightarrow\tx{Flat}(\C,\E) $$
	and that every $\Phi$-exact $\Phi_l\C\to\E$ is the left Kan extension of its restriction to $\C$. Thus we only need to prove that $\tx{Lan}_KF$ exists and is $\Phi$-exact whenever $F\colon\C\to\E$ is flat. 
	
	For that, let $F\colon\C\to\E$ be flat. Then, since $\Phi$ is small there exists a small full subcategory $J\colon\F\hookrightarrow\E$ that is closed under finite limits and $\Phi$-lex colimits, and such that $F$ factors as $F=J\circ F'$ with $F'\colon\C\to\F$. Then $\F$ is $\Phi$-exact (in our sense) by \cite[4.11 \& 4.2]{GL12:articolo}, and $F'$ is easily seen to be flat (since $J$ reflects limits and colimits in the codomain). Then $\tx{Lan}_KF'\colon\Phi_l\C\to\F$ exists and is $\Phi$-exact by $(2)$; moreover, $J\circ\tx{Lan}_KF'\cong \tx{Lan}_KF$ by the arguments at the beginning of $(3)\Rightarrow(2)$ above (using that $J$, being $\Phi$-exact, preserves $\Phi^*$-lex colimits). Thus $\tx{Lan}_KF$ is $\Phi$-exact being the composite of $\Phi$-exact $\V$-functors.
\end{proof}

\begin{obs}
	It is not possible to express free exact completions in the flat world in terms of existence of a left biadjoint to some forgetful functor (as in \cite[3.7]{GL12:articolo}). The problem being that flat $\V$-functors do not compose, not even if the categories involved satisfy the equivalent conditions of Theorem~\ref{main} above, and hence they cannot form the class of morphisms of a category.\\
	A formal explanation of why this happens in the case of exact completions is given in \cite[3.3]{CV98:articolo}. A direct counterexample can also be easily constructed in the setting of infinitary lextensive categories; that is, by considering the class of lex-colimits $\Phi_{\tx{ilext}}$ of Section~\ref{ilext}. The category $\bo{Fld}$ of fields is finitely accessible with connected limits; hence $\bo{Fld}\simeq\tx{Flat}(\bo{Fld}_f\op,\bo{Set})$ and $\bo{Fld}_f\op$ satisfies the conditions of the theorem above (by Proposition~\ref{ilext-prop}). However, the composition of any flat $\bo{Fld}_f\op\to\bo{Set}$, corresponding to a field $\mathbb K$, with the flat functor $\bo{Set}(2,-)\colon \bo{Set}\to\bo{Set}$ is not flat since it would correspond to a product $\mathbb K\times \mathbb K$ in $\bo{Fld}$, which does not exist.
\end{obs}

\begin{obs}
	Infinitary versions of regularity and Barr-exactness have been studied in the literature~\cite{Mak90:articolo} together with corresponding notions of free infinitary exact completions and of weak infinitary limits~\cite{HT96:articolo,vitale2001essential}. Similarly, flat functors also admit infinitary generalizations, which found important applications especially in the theory of accessible categories~\cite{Lai81:articolo,AR94:libro}.\\
	Both the infinitary notions of exact categories and flat functors are obtained by replacing the class of finite limits with that of $\lambda$-small limits, for a fixed regular cardinal $\lambda$. Such generalization could very likely be carried out also in our setting; however, we have decided not to address this here since it would require an infinitary version of~\cite{GL12:articolo}. One possible way to tackle this problem is to note that all the results of~\cite{GL12:articolo} and of this paper rely mostly on the fact that the class of finite limits is {\em sound} in the sense of~\cite{ABLR02:articolo,LT22:virtual}; thus an infinitary generalization is likely to succeed since class of $\lambda$-small limits (for a regular cardinal $\lambda$) is sound as well.
\end{obs}

\section{Examples}\label{examples}

In the sections below we consider specific examples of $\Phi$-exactness. In each case we express the necessary and sufficient conditions for a category to have a free $\Phi$-exact completion in terms of the existence of certain {\em virtual} (weak, multi, etc) finite limits. 

To show that, whenever a category has these specified virtual finite limits then it also has a free $\Phi$-exact completion, is quite standard and requires only some easy computations. However, to prove that the opposite implications holds we need to invoke the theory of accessible categories with limits; in all the examples we shall see that if the free $\Phi$-exact completion for $\C$ exists, then the category of flat functors from $\C$ to $\bo{Set}$ has limits of some sort, and then apply a duality theorem (as in \cite{Ten22:duality}) that implies the existence of said virtual limits.

\subsection{Weak limits, regularity, and exactness}\label{weaklylex}$ $

Let $\Phi_{\tx{reg}}$ and $\Phi_{\tx{ex}}$ be the classes of lex weights respectively for regular and Barr-exact categories (see \cite[5.1 and 5.2]{GL12:articolo}). A functor is $\Phi_{\tx{reg}}$ or $\Phi_{\tx{ex}}$-exact if and only if it is {\em regular}; meaning that it preserves finite limits and regular epimorphisms.

Recall the notion of weak limit below (see for instance \cite{Hu96:articolo,CV98:articolo,AR94:libro}).

\begin{Def}
	We say that a diagram $H\colon\D\to \C$ has a {\em weak limit} in $\C$ if there exists an object $C\in \C$ together with a cone $\delta\colon\Delta C\to H$ such that every other cone $e\colon\Delta E\to H$ factors as $e=\delta\circ f$ for some $f\colon E \to C$ in $\C$. We say that $\C$ is {\em weakly lex} if it has all weak finite limits.
\end{Def}

Equivalently, $C$ is a weak limit of $H$ if there exists a regular epimorphism $$\C(-,C)\twoheadrightarrow\lim YH$$ in $[\C\op,\bo{Set}]$, where $Y\colon\C\to[\C\op,\bo{Set}]$ is the Yoneda embedding. Note that weak limits are not unique.

\begin{prop}\label{free-exact}
	The following are equivalent for a small category $\C$:\begin{enumerate}\setlength\itemsep{0.25em}
		\item $\C$ has a free regular completion;
		\item $\C$ has a free exact completion;
		\item $\tx{Flat}(\C,\bo{Set})$ has products;
		\item $\C$ is weakly lex.
	\end{enumerate}
\end{prop}
\begin{proof}
	The implication $(1)\Rightarrow(2)$ is trivial by Theorem~\ref{main} since the inclusion $\Phi^\diamond_{\tx{reg}}[\C]\subseteq \Phi^\diamond_{\tx{ex}}[\C]$ always hold. For $(2)\Rightarrow(3)$, by hypothesis there is an equivalence $$\tx{Flat}(\C,\bo{Set})\simeq\tx{Reg}(\Phi_{\tx{ex}}(\C),\bo{Set}).$$ Since regular epimorphisms are stable in $\bo{Set}$ under products, we know that the category $\tx{Reg}(\Phi_{\tx{ex}}(\C),\bo{Set})$ is closed under products in the presheaf category; thus $\tx{Flat}(\C,\bo{Set})$ has products. Regarding $(3)\Rightarrow(4)$, the Cauchy completion $\hat\C$, of $\C$, is weakly lex by \cite[Corollary~1.6]{Hu96:articolo} (or \cite[Theorem~3.8]{Ten22:duality}) since it corresponds to the opposite of the category of the finitely presentable objects of $\tx{Flat}(\C,\bo{Set})$. Now, if $\hat A\in\hat \C$ is a weak limit of a finite diagram $H$ from $\C$ (seen in $\hat\C$); then, since we can write $\hat A$ as a split quotient of some object $A$ of $\C$, it is easy to see that $A$ is actually a weak limit of $H$ in $\C$. Thus $\C$ is weakly lex.
	
	$(4)\Rightarrow(1)$. Consider the full subcategory $\D$ of $[\C\op,\bo{Set}]$ spanned by those functors $M\colon\C\op\to\bo{Set}$ which arise as
	$$ YA\stackrel{q}{\twoheadrightarrow} M \stackrel{m}{\rightarrowtail} YB_1\times\cdots \times YB_n $$
	with $q$ a regular epimorphism and $m$ a monomorphism. Then $\D$ contains the representables as well as any finite limit of them; indeed, any finite limit of representables is covered by a representable since $\C$ is weakly lex, and is a subobject of a finite product of representables by finiteness of the limit itself. To conclude it is then enough to prove that $\D\subseteq \Phi^\diamond_{\tx{reg}}[\C]$ and apply Theorem~\ref{main}.\\
	Consider $M\in\D$ and a flat functor $F\colon \C\to \E$ into a regular category $\E$; we need to prove that $M*F$ exists in $\E$ and is preserved by any regular functor out of $\E$. Given a presentation for $M$ as above, we can see $M$ as the coequalizer of the kernel pair of $mq\colon YA\to YB_1\times\cdots \times YB_n$. By fully faithfulness of the Yoneda embedding there exist maps $h_i\colon A\to B_i$ in $\C$ for which $(Yh_1,\cdots,Yh_n)=mq$. Now, it follows that the kernel pair of $mq$ can be rewritten as a finite limit of representables; since $F$ is flat, the colimit functor $-*F$ preserves that limit and hence $M*F$ corresponds to the coequalizer of the kernel pair of $(Fh_1,\cdots,Fh_n)\colon FA\to FB_1\times\cdots\times FB_n$. Such colimit exists in any regular category and is preserved by any regular functor.
\end{proof}

Note that the category $\D$ described above is regular whenever $\C$ is weakly lex (by \cite[Proposition~9]{CV98:articolo}); since it also contains the representables and is contained in $\Phi^\diamond_{\tx{reg}}[\C]$, then we necessarily have equalities $$\D=\Phi^\diamond_{\tx{reg}}[\C]= \Phi_{\tx{reg}}(\C);$$
showing that $\D$ is the free regular completion of $\C$.

\begin{obs}
	Proposition~\ref{free-exact} above can be extended to the setting of categories enriched over a symmetric monoidal finitary variety $\V$ as in \cite{LT20:articolo}. The notions of regularity and exactness corresponding to points (1) and (2) are those introduced in \cite[Section~5]{LT20:articolo}; the limits involved in (3) are products and powers by regular projective objects; while the notion of enriched weak limit to consider in (4) is that appearing in \cite[Section~5]{LT22:limits} (where $\E$ and $\G$ are as in \cite[Example~4.47]{LT22:limits}). The proof extends without major modifications using the duality \cite[Theorem~4.24]{Ten22:duality}. 
\end{obs}

To conclude that our regular and exact completions have the same universal properties as the regular and exact completions for a weakly lex category considered in \cite{Hu96:articolo,CV98:articolo}, we need to show that our notion of flatness coincides with Carboni and Vitale's notion of left covering functor, when the domain category is weakly lex. 

Let us first recall the definition, which was first given in Vitale's PhD thesis \cite{vitale1994left}.

\begin{Def}
	Let $F\colon\C\to\E$ be a functor from a weakly lex category $\C$ to a regular category $\E$. We say that $F$ is {\em left covering} if for any finite diagram $H\colon\D\to \C$ and any weak limit $C\in\C$ of $H$, the comparison map
	$$ FC\twoheadrightarrow \lim (FH) $$%
	is a regular epimorphism.
\end{Def}

These were also defined independently by Hu, who called them {\em $\omega$-flat} in \cite{Hu96:articolo, hu1996note}, following a notation closer to ours.

\begin{prop}
	Let $\C$ be small and weakly lex. A functor $F\colon\C\to\E$ into a regular category $\E$ is flat if and only if it is left covering.
\end{prop}
\begin{proof}
	By the proposition above $\C$ has a free regular category which can be identified with the full subcategory $\D$ of $[\C\op,\bo{Set}]$ just described. This also coincides with the free regular completion of the weakly lex $\C$ in the sense of \cite{CV98:articolo}. Thus flat functors out of $\C$ and left covering functors coincide since they both can be described as those whose left Kan extension along the inclusion $\C\hookrightarrow\D$ is regular (by definition and by \cite[Theorem~29]{CV98:articolo}).
\end{proof}

Thus we are able to capture regular and exact completions with respect to weakly lex categories as part of this framework.

\subsection{Multilimits and infinitary lextensive categories}\label{ilext}$ $

Let $\Phi_{\tx{ilext}}$ be the classes of lex weights for infinitary lextensive categories (see the infinite version of \cite[5.3]{GL12:articolo}). In this setting $\Phi_{\tx{ilext}}$-lex colimits are just coproducts; thus a functor is $\Phi_{\tx{ilext}}$-exact if and only if it is lex and preserves coproducts. Given infinitary lextensive categories $\E$ and $\F$, denote by $\tx{iLext}(\E,\F)$ the full subcategory of $[\E,\F]$ spanned by the $\Phi_{\tx{ilext}}$-exact functors.

We recall the notion of multilimit below:

\begin{Def}[\cite{Diersmultialge}]
	We say that a diagram $H\colon\D\to \C$ has a {\em multilimit} in $\C$ if there exists a family of objects $(C_i)_{i\in I}$ in $\C$ together with cones $\delta_i\colon\Delta C_i\to H$ such that every other cone $e\colon\Delta E\to H$ factors as $e=\delta_i\circ f$ for a unique $i\in I$ and a unique $f\colon E \to C_i$ in $\C$. We say that $\C$ has {\em finite multilimits} if it has all multilimits indexed by finite diagrams.
\end{Def}

Unlike weak limits, multilimits are unique up to isomorphism. Moreover, a family $(C_i)_{i\in I}$ is the multilimit of $H$ if and only if there exists an isomorphism $$\sum_i\C(-,C_i)\cong\lim YH$$ in $[\C\op,\bo{Set}]$, where $Y\colon\C\to[\C\op,\bo{Set}]$ is the Yoneda embedding.

In the Proposition below Cauchy completeness is required for the implication $(2)\Rightarrow(3)$ to be true. This additional hypothesis appears here, but not in Proposition~\ref{free-exact}, since the class $\Phi_{\tx{ex}}$ considered before subsumes absolute coequalizers (that is, all Cauchy colimits) while $\Phi_{\tx{ilext}}$ does not.

\begin{prop}\label{ilext-prop}
	The following are equivalent for a small Cauchy complete category $\C$:\begin{enumerate}\setlength\itemsep{0.25em}
		\item $\C$ has a free infinitary lextensive completion;
		\item $\tx{Flat}(\C,\bo{Set})$ has connected limits;
		\item $\C$ has finite multilimits.
	\end{enumerate}
\end{prop}
\begin{proof}
	For $(1)\Rightarrow(2)$, if the free infinitary lextensive completion of $\C$ exists then $$\tx{Flat}(\C,\bo{Set})\simeq\tx{iLext}(\Phi_{\tx{ilext}}(\C),\bo{Set})$$ has connected limits since these commute in $\bo{Set}$ with coproducts. Regarding $(2)\Rightarrow(3)$, the category $\C$ has finite multilimits by Diers duality for locally finitely multipresentable categories \cite{Die80:articolo} which requires Cauchy completeness (see also \cite[Theorem~3.8]{Ten22:duality}).
	
	$(3)\Rightarrow(1)$. Let $\tx{Fam}(\C)$ be the closure of the representables in $[\C\op,\bo{Set}]$ under small coproducts (that is, the coproduct completion of $\C$); then $\tx{Fam}(\C)\subseteq \Phi_{\tx{ilext}}^\diamond[\C]$ since, independently of the colimits being along a flat diagram, coproducts exist in any infinitary lextensive category and are preserved by any infinitary lextensive functor by definition.
	
	Thus, if $\C$ has finite multilimits then $\Phi_{\tx{lext}}^\diamond[\C]$ has finite limits of representables and the free infinitary lextensive completion exists by theorem~\ref{main}.
\end{proof}

Note that then $\tx{Fam}(\C)$ is the free infinitary lextensive completion of $\C$. In fact, whenever $\C$ has finite multilimits then $\tx{Fam}(\C)$ is lex (see for instance \cite[Corollary~3.9]{Ten22:duality}), and hence an infinitary lextensive category.

\begin{Def}[\cite{Diersmultialge}]
	Let $F\colon\C\to\E$ be a functor from a category $\C$ with finite multilimits into a lex category $\E$ with coproducts. We say that $F$ is {\em finitely multicontinuous} if for any finite diagram $H\colon \D\to\C$ with multilimit $(C_i)_{i\in I}$ there is an isomorphism
	$$ \sum_{i\in I}FC_i\cong \lim FH $$%
	induced by the image of the limiting cones.
\end{Def}

\begin{prop}\label{flat-multi}
	Let $\C$ be small and have finite multilimits. A functor $F\colon\C\to\E$ into an infinitary lextensive category $\E$ is flat if and only if it is finitely multicontinuous.
\end{prop}
\begin{proof}
	By Proposition~\ref{flat-prop}, $F$ is flat if and only if for any finite diagram $H\colon\D\to \C$ there is an isomorphism
	$$ \lim FH\cong  (\lim YH)*F $$%
	Now, since $\C$ has finite multilimits, $\lim YH\cong \sum_i\C(-,C_i)$ with $(C_i)_{i\in I}$ the multilimit of $H$; therefore 
	$$ (\lim YH)*F \cong \sum_i\C(-,C_i)*F\cong \sum_i FC_i. $$%
	This shows that $F$ is flat if and only if it is finitely multicontinuous.
\end{proof}

\begin{obs}
	All this can be captured in the context of Section~\ref{Psi-virtual}, where we discuss generalizations to the enriched context.
\end{obs}

\subsection{Fc-limits and pretopoi}\label{fc}$ $

Let $\Phi_{\tx{pret}}=\Phi_{\tx{ex}}\cup \Phi_{\tx{lext}}$ be the classes of lex weights for pretopoi. A functor is $\Phi_{\tx{pret}}$-exact if and only if it is lex and preserves finite coproducts and regular epimorphisms; these are usually called {\em models}. Given pretopoi $\E$ and $\F$, denote by $\tx{Mod}(\E,\F)$ the full subcategory of $[\E,\F]$ spanned by the models of $\E$ in $\F$.

The notion of fc-limit was first considered in \cite[Definition~3.1]{karazeris2005completeness} (with the name fm-limit), and can be thought as a generalization of that of weak limit and multi-finite limit. See also \cite{beke2005flatness} where fc-limits and colimits are used in the context of flat functors.

\begin{Def}[\cite{karazeris2005completeness}]
	We say that a diagram $H\colon\D\to \C$ has a {\em fc-limit} in $\C$ if there exists a finite family of objects $(C_i)_{i\leq n}$ in $\C$ together with cones $\delta_i\colon\Delta C_i\to H$ such that every other cone $e\colon\Delta E\to H$ factors as $e=\delta_i\circ f$ for some $i\in I$ and $f\colon E \to C_i$ in $\C$. We say that $\C$ has {\em finite fc-limits} if it has all fc-limits indexed by finite diagrams.
\end{Def}

Equivalently, $(C_i)_{i\leq n}$ is a fc-limit of $H$ if there exists a regular epimorphism $$\sum_{i\leq n}\C(-,C_i)\twoheadrightarrow\lim YH$$ in $[\C\op,\bo{Set}]$. Like weak limits, fc-limits are not uniquely determined by the universal property. 

The existence of finite fc-limit is what allows to consider free pretopos completions. Note that, by Corollary~\ref{flat-ultraproducts}, the ultraproducts of $\tx{Flat}(\C,\bo{Set})$ in point $(2)$ below satisfy an universal property in the sense of Definition~\ref{univ-ultra}. 

As we outline in the proof, the equivalence of $(2)$ and $(3)$ was first shown in \cite{beke2005flatness}, while $(3)\Rightarrow(1)$ is \cite[Theorem~5.3]{kar04}.

\begin{prop}\label{free-pretopos}
	The following are equivalent for a small category $\C$:\begin{enumerate}\setlength\itemsep{0.25em}
		\item $\C$ has a free pretopos completion;
		\item $\tx{Flat}(\C,\bo{Set})$ is closed under ultraproducts in $[\C,\bo{Set}]$;
		\item $\C$ has finite fc-limits.
	\end{enumerate}
\end{prop}
\begin{proof}
	$(1)\Rightarrow(2)$. Let $\Phi_{\tx{pret}}(\C)$ be the free pretopos completion of $\C$; then  $$\tx{Flat}(\C,\bo{Set})\simeq\tx{Mod}(\Phi_{\tx{ex}}(\C),\bo{Set}).$$ In particular,  $\tx{Flat}(\C,\bo{Set})$ can be identified with a full subcategory of $\K:=[\Phi_{\tx{ex}}(\C),\bo{Set}]$ closed under ultraproducts. Since restricting along the inclusion $\C\hookrightarrow\Phi_{\tx{pret}}(\C)$ preserves all limits and colimits, it also preserves ultraproducts. Therefore $\tx{Flat}(\C,\bo{Set})$ is closed under ultraproducts $[\C,\bo{Set}]$.

	$(2)\Rightarrow(3)$. Let $\K=[\C,\bo{Set}]$, so that $\K_f$ coincides with the free cocompletion of $\C\op$ under finite colimits. Then the fact that $(2)$ implies $(3)$ is given by putting together Theorem~4.2 (with $\A=\K_f$) and Theorem~3.3 of \cite{beke2005flatness}.

	$(3)\Rightarrow(1)$. We generalize the argument in the proof of Proposition~\ref{free-exact}. Assume that $\C$ has finite fc-limits and consider the full subcategory $\D$ of $[\C\op,\bo{Set}]$ spanned by those functors $M\colon\C\op\to\bo{Set}$ which arise as
	$$ YA_1+\cdots+YA_n\stackrel{q}{\twoheadrightarrow} M \stackrel{m}{\rightarrowtail} YB_1\times\cdots \times YB_m $$
	with $q$ a regular epimorphism and $m$ a monomorphism. Then $\D$ contains the representables and and any finite limit of them; indeed, any finite limit of representables is covered by a finite coproduct of representable since $\C$ is has finite fc-limits, and is a subobject of a finite product of representables by finiteness of the limit itself. To conclude we need to show that $\D\subseteq \Phi^\diamond_{\tx{pret}}[\C]$, so that we can apply Theorem~\ref{main}.\\
	Consider $M\in\D$ and a flat functor $F\colon \C\to \E$ into a pretopos $\E$; we need to prove that $M*F$ exists in $\E$ and is preserved by any pretopos functor out of $\E$. Given a presentation for $M$ as above, we can see $M$ as the coequalizer of the kernel pair of $mq\colon \textstyle\sum_{i}YA_i\to\textstyle\prod_{j}YB_j$. Such kernel pair, being computed pointwise, is isomorphic to the coproduct of the kernel pairs of each $i$-component $(mq)_i\colon YA_i\to\textstyle\prod_{j}YB_j$, which in turn can be seen as a finite limit of representables involving maps $h_{i,j}\colon A_i\to B_j$. Now, $-*F$ preserves all colimits and any finite limit of representables (since $F$ is flat); thus $M*F$ can be computed as the coproduct, indexed on $i\leq n$, of the coequalizer of the kernel pairs of $(Fh_{i,1},\cdots,Fh_{i,n})\colon FA_i\to \textstyle\prod_j FB_j$. Such combination of limits and colimits exists in any pretopos and are preserved by any pretopos functor by definition.
\end{proof}

\begin{obs}\label{fc-refl}
	Following \cite[Section~2]{beke2005flatness}, we say that a fully faithful inclusion $J\colon\A\hookrightarrow\K$ is {\em fc-reflective} if for any $K\in\K$ there exists a finite family $(A_i)_{i\leq n}$ in $\A$ together with a regular epimorphism
	$$\sum_{i\leq n}\A(A_i,-)\twoheadrightarrow\K(K,J-).$$
	Then \cite[3.3]{beke2005flatness} says that the conditions of the proposition above are further equivalent to 
	\begin{enumerate}
		\item[\em (4)] {\em $\C\op$ is fc-reflective in $\tx{Fin}(\C\op)$, the free cocompletion of $\C\op$ under finite colimits.} 
	\end{enumerate}
	In general the inclusion $\tx{Flat}(\C\op,\bo{Set})\hookrightarrow[\C\op,\bo{Set}]$ is not fc-reflective \cite[Remark~4.6]{beke2005flatness}.
\end{obs}

It remains to characterize flat functors out of a category $\C$ with finite fc-limits. 

\begin{Def}[\cite{kar04}]
	Let $F\colon\C\to\E$ be a functor from category $\C$ with finite fc-limits to a pretopos $\E$. We say that $F$ is {\em finitely fc-continuous} if for any finite diagram $H\colon\D\to \C$ and any fc-limit $(C_i)_{i\leq n}$ of $H$ in $\C$, the comparison map
	$$ \sum_{i\leq n} FC_i\twoheadrightarrow  \lim (FH) $$%
	is a regular epimorphism.
\end{Def}

\begin{obs}
	The definition above was first considered in \cite[Section~5]{kar04} to characterize flat functors into a pretopos $\E$ with respect to the precanonical topology on it. See Remark~\ref{kar}
\end{obs}

\begin{prop}
	Let $\C$ be small and have finite fc-limits. A functor $F\colon\C\to\E$ into a pretopos $\E$ is flat if and only if it is finitely fc-continuous.
\end{prop}
\begin{proof}
	Let $\D$ be the full subcategory of $[\C\op,\bo{Set}]$ described in the proof above, and let $J\colon\C\to\D$ be the inclusion. Arguing as in \cite{CV98:articolo}, it is easy to see that $\D$ is closed in $[\C\op,\bo{Set}]$ under finite limits, and that if $F$ is finitely fc-continuous then $\tx{Lan}_JF$ is left covering. Thus $\tx{Lan}_JF$ is lex by \cite[Proposition~20]{CV98:articolo} and hence, since $\C_l\subseteq \D$, $F$ is flat.
	
	Conversely, assume that $F\colon\C\to\E$ is flat. For any finite diagram $H\colon\D\to \C$, and any fc-limit $(C_i)_{i\leq n}$ of $H$ we have a regular epimorphism $\textstyle\sum_{i\leq n}\C(-,C_i)\twoheadrightarrow\lim YH$. Applying the functor $-*F$ we obtain a regular epimorphism
	$$ \sum_{i\leq n}FC_i\twoheadrightarrow\lim YH*F. $$%
	Composing this with the isomorphism $(\lim YH)*F\cong \lim FH$ given by flatness of $F$, we obtain that $F$ is finitely fc-continuous.
\end{proof}

\subsection{Multi-finite limits and lextensive categories}\label{lextesive-sect}$ $

Let $\Phi_{\tx{lext}}$ be the classes of lex weights for lextensive categories (see \cite[5.3]{GL12:articolo}). In this setting $\Phi_{\tx{lext}}$-lex colimits are just finite coproducts; thus a functor is $\Phi_{\tx{lext}}$-exact if and only if it is lex and preserves finite coproducts. Given lextensive categories $\E$ and $\F$, denote by $\tx{Lext}(\E,\F)$ the full subcategory of $[\E,\F]$ spanned by the $\Phi_{\tx{lext}}$-exact functors.

First let us introduce the finite-indexed version of multilimits:

\begin{Def}
	We say that a diagram $H\colon\D\to \C$ has a {\em multi-finite limit} in $\C$ if there exists a finite family of objects $(C_i)_{i\leq n}$ in $\C$ together with cones $\delta_i\colon\Delta C_i\to H$ which form the multilimit of $H$ in $\C$. We say that $\C$ has {\em finite multi-finite limits} if it has all multi-finite limits indexed by finite diagrams.
\end{Def}

Since multi-finite limits are in particular multilimits, a finite family $(C_i)_{i\leq n}$ is the multi-finite limit of $H$ if and only if there exists an isomorphism $$\sum_{i\leq n}\C(-,C_i) \cong\lim YH$$ in $[\C\op,\bo{Set}]$.

Cauchy completeness below needs to be assumed for the same reasons explained above Proposition~\ref{ilext-prop}.

\begin{prop}\label{free-lext}
	The following are equivalent for a small Cauchy complete category $\C$:\begin{enumerate}\setlength\itemsep{0.25em}
		\item $\C$ has a free lextensive completion;
		\item $\tx{Flat}(\C,\bo{Set})$ is closed under connected limits and ultraproducts in $[\C,\bo{Set}]$;
		\item $\C$ has finite multi-finite limits.
	\end{enumerate}
\end{prop}
\begin{proof}
	$(1)\Rightarrow (2)$. If the free lextensive completion of $\C$ exists then also the free infinitary lextensive completion exists (by \cite[Theorem~7.7]{GL12:articolo}); thus $\tx{Flat}(\C,\bo{Set})$ has connected limits by Proposition~\ref{ilext-prop}. Now, since also the free pretopos completion exists in this case (by the same theorem of \cite{GL12:articolo}), the category $\tx{Flat}(\C,\bo{Set})$ is also closed under ultraproducts in $[\C,\bo{Set}]$ by Proposition~\ref{free-pretopos}, showing $(2)$.
	
	$(2)\Rightarrow (3)$. Again by Propositions~\ref{ilext-prop} and~\ref{free-pretopos}, $\C$ has both finite multilimits and finite fc-limits. The existence of these implies that the finite multilimits of $\C$ are actually multi-finite.

	$(3)\Rightarrow (1)$. Let $\tx{Fam}_f(\C)$ be the closure of the representables in $[\C\op,\bo{Set}]$ under finite coproducts (that is, the finite coproduct completion of $\C$); then it is easy to see that $\tx{Fam}_f(\C)\subseteq \Phi_{\tx{lext}}^\diamond[\C]$ for any $\C$. Thus, if $\C$ has finite multi-finite limits then $\Phi_{\tx{lext}}^\diamond[\C]$ has finite limits of representables and the free lextensive completion exists by Theorem~\ref{main}.
\end{proof}

Note that then $\tx{Fam}_f(\C)$ is the free lextensive completion of $\C$. In fact, it is easy to see that if $\C$ has finite multi-finite limits then $\tx{Fam}_f(\C)$ is lex, and hence a lextensive category.

The following definition is a variation of the notion of finitely multicontinuous functor:

\begin{Def}
	Let $F\colon\C\to\E$ be a functor from a category $\C$ with finite multi-finite limits into a lex category $\E$ with finite coproducts. We say that $F$ {\em merges finite multi-finite limits} if for any finite diagram $H\colon \D\to\C$ with multi-finite limit $(C_i)_{i\leq n}$ there is an isomorphism
	$$ \sum_{i\leq n}FC_i\cong \lim FH $$%
	induced by the image of the limiting cones.
\end{Def}

\begin{prop}
	Let $\C$ be small and have finite multi-finite limits. A functor $F\colon\C\to\E$ into a lextensive category $\E$ is flat if and only if it merges finite multi-finite limits.
\end{prop}
\begin{proof}
	Argue exactly as in the proof of Proposition~\ref{flat-multi}.
\end{proof}

\subsection{Polylimits}\label{poly}$ $

Let us now define a class $\Phi$ of lex colimits that wasn't considered in \cite{GL12:articolo}; this will describe the {\em quasi-based} lex categories of \cite{HT96qc:articolo} as free $\Phi$-exact completions. 

Let $\C$ be a category with an initial object. Following \cite[Definition~4.10]{Ten22:duality}, we say that a diagram $H\colon \G \to\C$, indexed on a small groupoid $\G$, is a free action if for each $g\neq h\colon x\to y$ in $\G$ the equalizer of $(Hg,Hh)$ is the initial object of $\C$. We say that $\C$ has colimits of free actions if it has colimits of diagrams taken along a free action. Note that every diagram from a discrete groupoid $\G$ is free, hence coproducts are colimits of free actions.

Let $\Phi$ be the class of lex weights $M\colon C\op\to\bo{Set}$, where $\C$ is lex and $M\cong\colim (YH)$ with $H\colon\G\to\C$ a groupoid-indexed diagram for which $YH$ is a free action in $[\C\op,\bo{Set}]$. In particular, by the arguments above, $\Phi$ contains the class $\Phi_{\tx{ilext}}$ of Section~\ref{ilext}.

\begin{obs}\label{freecompanion}
	It follows that the full subcategory $\Phi[\C]$ of $[\C\op,\bo{Set}]$ spanned by the weights in $\Phi$ with domain $\C\op$, coincides with the category $\mt F_1\C$ considered in \cite{LT22:limits,Ten22:duality}.
\end{obs}

\begin{lema}\label{freeact}
	If a lex category $\E$ has $\Phi$-lex colimits then it has colimits of free actions. If $\E$ is $\Phi$-exact, $M\colon\C\op\to\bo{Set}$ is in $\Phi$, and $K\colon\C\to\E$ is lex, then
	$$M*K\cong \colim KH$$
	for some $H\colon\G\to\C$ for which $KH$ is a free action in $\E$.
	In particular a lex functor $F\colon\E\to\F$, out of a $\Phi$-exact $\E$, preserves $\Phi$-lex colimits if and only if it preserves colimits of free actions.
\end{lema}
\begin{proof}
	Suppose that $\E$ has $\Phi$-lex colimits; then in particular it has all coproducts and hence an initial object. Now, for any free action $H\colon \G \to\E$, the weight $M:=\colim (YH)$ in $[\E\op,\bo{Set}]$ is by definition in $\Phi$. Thus the $\Phi$-lex colimit $M*1_\E$ exists in $\C$ by hypothesis. But $M*1_\E\cong\colim H$; therefore $\E$ has colimits of free actions.
	
	Assume now that $\E$ is $\Phi$-exact and consider $M\colon\C\op\to\bo{Set}$ in $\Phi$ and a lex $K\colon\C\to\E$. Since by definition $M\cong\colim (Y_{\C}H)$ with $H\colon\G\to\C$ and $Y_{\C}H$ a free action, we have an isomorphism $M*K\cong\colim KH$; thus to conclude it is enough to prove that $KH$ is a free action in $\E$.
	
	Consider the isomorphism $Y_{\E}.K.H\cong \P K.Y_{\C}.H$, where $\P K\colon\P\C\to\P\E$ is the cocontinuous functor induced by taking free cocompletions. Since $\P K$ is lex by \cite[Remark~6.6]{DL07} (because $K$ was), it follows that $Y_{\E}KH$ is still a free action (since $Y_{\C}H$ was). Now, by \cite[Proposition~3.4]{GL12:articolo} the inclusion $J\colon\E\to\Phi_l\E$ has a lex left adjoint $T\colon\Phi_l\E\to\E$. Since $\Phi_l\E$ is closed in $\P\E$ under finite limits and the initial object, the functor $JKH\colon\C\to \Phi_l\E$ (obtained by restricting the codomain of $Y_{\C}KH$) is a free action. Then, since $T$ is lex and preserves (any existing colimit and in particular) the initial object of $\Phi_l\E$, the diagram $KH\cong TJKH$ is also a free action, as desired.
\end{proof}

As a consequence we can characterize $\Phi$-exact categories as follows:

\begin{cor}
	A category $\E$ is $\Phi$-exact if and only if it can be identified with a full subcategory of a small-exact $\T$ closed under finite limits and colimits of free actions.
\end{cor}
\begin{proof}
	If $\E$ is $\Phi$-exact, this follows by Definition~\ref{exact} and by the lemma above. Conversely, assume that $J\colon\E\to\T$ is a fully faithful embedding of $\E$ in a small-exact $\T$, that $\E$ has finite limits and colimits of free actions, and that $J$ preserves them. It is enough to prove that $\E$ is also closed under $\Phi$-lex colimits. 
	
	Given $M\colon\C\op\to\bo{Set}$ in $\Phi$ and a lex $K\colon\C\to\E$, since every small-exact is $\Phi$-exact, the lemma above implies that $M*JK\cong \colim JKH$ for some $H\colon\G\to\C$ for which $JKH$ is a free action. But $J$ is fully faithful, lex, and preserves the initial object; thus $KH$ is also a free action in $\E$. It follows that $\colim KH$ exists in $\E$ and is preserved by $J$. Again by fully faithfulness, this implies that $M*K\cong \colim KH $ exists in $\E$ and is preserved by $J$.
\end{proof}

The virtual limits involved in this kind of free exact completions are polylimits:

\begin{Def}{\cite[Definition~0.12]{Lam89:PhD}}
	Let $H\colon\C\to\A$ be a diagram in a category $\A$. A {\em polylimit} of $H$ is given by a family of cones $(c_i\colon\Delta A_i\to H)_{i\in I}$, with $A_i\in\A$, with the following property: for any cone $c\colon\Delta C\to H$ there exists a unique $i\in I$ and a map $f\colon C\to A_i$ in $\A$ such that $c=c_i\circ \Delta f$, moreover $f$ is unique up to unique automorphism of $A_i$.
\end{Def}

We are now ready to prove the following. Once again Cauchy completeness needs to be assumed for the same reasons explained above Proposition~\ref{ilext-prop}. By {\em wide pullbacks} we mean limits of diagrams of the form $(h_i\colon A_i\to B)_{i\in I}$ for a possibly infinite set $I$.

\begin{prop}
	The following are equivalent for a small Cauchy complete category $\C$: \begin{enumerate}\setlength\itemsep{0.25em}
		\item $\C$ has a free $\Phi$-exact completion;
		\item $\tx{Flat}(\C,\bo{Set})$ has wide pullbacks;
		\item $\C$ has finite polylimits.
	\end{enumerate}
\end{prop}
\begin{proof}
	$(1)\Rightarrow(2)$. Let $\Phi_l\C$ be the free $\Phi$-exact completion of $\C$; then  $$\tx{Flat}(\C,\bo{Set})\simeq\Phi\tx{-Ex}(\Phi_l\C,\bo{Set}).$$ 
	By Lemma~\ref{freeact} above, a functor $\Phi_l\C\to\bo{Set}$ is $\Phi$-exact if and only if it preserves finite limits and colimits of free actions. Since these commute in $\bo{Set}$ with wide pullbacks (see \cite[Proposition~1.4]{HT96qc:articolo}, where colimits of free actions are called quasi-coproducts), the category $\Phi\tx{-Ex}(\Phi_l\C,\bo{Set})$ is closed in the presheaf category under wide pullbacks. Thus $(2)$ follows.
	
	$(2)\Rightarrow(3)$. This is a direct consequence of the duality theorem of \cite{HT96qc:articolo} which requires $\C$ to be Cauchy complete --- see also \cite[Proposition~4.28]{LT22:limits} and \cite[Section~4.2]{Ten22:duality}.
	
	$(3)\Rightarrow(1)$. If $\C$ has finite polylimits the full subcategory $\D:=\mt F_1\C$ of $[\C\op,\bo{Set}]$, spanned by colimits of free actions of representables in $[\C\op,\bo{Set}]$ is closed under finite limits and colimits of free actions (by \cite[4.28]{LT22:limits} and \cite[3.9]{Ten22:duality}) making it $\Phi$-exact. Let $K\colon\C\hookrightarrow\D$ be the inclusion; it is now standard to see that a functor $F\colon\C\to\E$, into a $\Phi$-exact $\E$, is flat if and only if $\tx{Lan}_KF$ is lex and preserves colimits of free actions. Therefore $\D$ is the free $\Phi$-exact completions of $\C$ (and coincides with $\Phi_l\C$). 
\end{proof}

It follows from the proof, that $\E$ is the free $\Phi$-exact completion of a small category $\C$ with finite polylimits, if and only if $\E\simeq \mt F_1\C$ and is lex, if and only if it is quasi-based in the sense of \cite[Section~3]{HT96qc:articolo} (thanks to \cite[4.2.1]{LT22:limits}).

As with the other examples, we can characterize flat functors out of a category $\C$ with finite polylimits. In fact, an easy computation shows that $F\colon\C\to\E$, into a $\Phi$-exact $\E$, is flat if and only if it is finitely polycontinuous in the sense of \cite[2.3]{HT96qc:articolo}.

\subsection{$\Psi^+$-virtual limits}\label{Psi-virtual}$ $

For this section we fix a {\em weakly sound} class of weights $\Psi$ as defined in \cite[4.4]{LT22:limits}, which generalizes the original concept of \cite{ABLR02:articolo}. In short, $\Psi$ is weakly sound if any small and $\Psi$-continuous $\V$-functor $M\colon\A\to\V$ (from a $\Psi$-complete $\A$ for which the free limit completion $\P^\dagger\A$ is complete) is {\em $\Psi$-flat}, meaning that $M$-weighted colimits commute in $\V$ with $\Psi$-limits.

For a list of examples see \cite[4.8]{LT22:limits}. Denote by $\Psi^+$ be the corresponding class of weights, with small domain, whose colimits commute in $\V$ with $\Psi$-limits. 

We consider the class of lex weights $\Phi$ induced by $\Psi$ as follows: for any weight $M\colon\C\op\to\V$ with small domain, consider the inclusion $J\colon\C\to\tx{Fin}(\C)$ of $\C$ into its free completion under finite colimits, then we let $\tx{Lan}_{J\op}M\in\Phi$ if and only if $M\in\Psi^+$.

In this setting $\Phi$-lex colimits in a lex $\V$-category are just $\Psi^+$-colimits, and a $\Phi$-exact $\V$-functor is one that preserves finite limits and $\Psi^+$-colimits. This is because, given $M\in\Psi^+$ as above and $H\colon\C\to\E$ into a lex $\V$-category $\E$, the $\Psi^+$-colimit $M*H$ exists if and only if the $\Phi$-lex colimit $$\tx{Lan}_{J\op}M*\tx{Ran}_JH$$ exists, and in that case they coincide (here $\tx{Ran}_JH$ exists and is lex by definition of $J$).

The notion of limit below was first considered in \cite[Definition~4.3]{karazeris2009representability} and studied more recently in \cite{LT22:limits,Ten22:duality}.

\begin{Def}
	Given a $\V$-category $\C$, a weight $M\colon\D\to\V$, and $H\colon\D\to\C$; we say that the {\em $\Psi^+$-virtual limit} of $H$ weighted by $M$ exists in $\C$ if $$[\D,\V](M,\C(-,H))\colon\C\op\to\V$$ lies in $\Psi^+(\B)$. 
	We say that $\B$ is {\em finitely $\Psi^+$-virtually complete} if it has all $\Psi^+$-virtual limits weighted by a finite weight $M$.
\end{Def}

Since $[\C,\V](M,\B(-,H))\cong \{M,YH\}$, the $\Psi^+$-virtual limit in $\C$ exists if and only if the $\V$-functor $\{M,YH\}$ in $[\C\op,\V]$ is a $\Psi^+$-colimit of representables. Then our characterization Theorem~\ref{main} implies:

\begin{prop}
	The following are equivalent for a small and Cauchy complete $\V$-category $\C$:\begin{enumerate}\setlength\itemsep{0.25em}
		\item $\C$ has a free $\Phi$-exact completion;
		\item $\tx{Flat}(\C,\V)$ has $\Psi$-limits;
		\item $\C$ is finitely $\Psi^+$-virtually complete.
	\end{enumerate}
	In that case then $\Phi_l(\C)=\Psi^+(\C)$.
\end{prop}
\begin{proof}
	If the free $\Phi$-exact completion of $\C$ exists then $$\tx{Flat}(\C,\V)\simeq\Phi\tx{-Ex}(\Phi_l(\C),\V)$$ and has $\Psi$-limits since these commute is $\V$ with $\Psi^+$-colimits. Thus $(1)\Rightarrow(2)$ holds. For $(2)\Rightarrow(3)$ the $\V$-category $\C$ is finitely $\Psi^+$-virtually complete by \cite[Theorem~3.8]{Ten22:duality} (applied to $\mt C=\Psi^+$).
	
	$(3)\Rightarrow(1)$. The free cocompletion under $\Psi^+$-colimits $\Psi^+(\C)$ is contained in $\Phi^\diamond[\C]$ since, independently of the colimits being taken along a flat diagram, $\Psi^+$-colimits exist in any $\Phi$-exact $\V$-category and are preserved by any $\Phi$-exact $\V$-functor by definition. Thus, if $\C$ is finitely $\Psi^+$-virtually complete then $\Phi_{\tx{lext}}^\diamond[\C]$ has finite limits of representables and the free $\Phi$-exact completion exists by Theorem~\ref{main}. 
	
	The assertion that $\Phi_l(\C)=\Psi^+(\C)$ follows from the fact that $\Psi^+(\C)$ is contained in $\Phi_l(\C)$, has $\Phi$-lex colimits (by definition) and is lex by \cite[Corollary~3.9]{Ten22:duality}.
\end{proof}

\begin{es}
	The case where $\V=\bo{Set}$ and $\Psi$ be the weakly sound class of finite limits, so that $\Phi$-lex colimits are just filtered colimits; then $\Phi$-exactness has been studied in \cite[Section~5.9]{GL12:articolo}.
\end{es}

\begin{es}
	Let $\V=\bo{Cat}$ and $\Psi$ be the weakly sound class generated by connected conical limits and powers by connected categories \cite[4.8(11)]{LT22:limits}, so that $\Phi$-lex colimits are given by small coproducts (and splitting of idempotents). In this case $\Phi$-exactness determines a 2-categorical generalization of the infinitary lextensive categories from Section~\ref{ilext}. 
\end{es}

\section{Universal ultraproducts}\label{universalultra}

Given a full subcategory $\M$ of a presheaf category $[\C,\bo{Set}]$ which is closed under the construction of categorical ultraproducts (see below), it is not true in general that these satisfy an universal property in $\M$ (see~\cite{HN81}). The aim of this section is to prove that instead the ultraproducts appearing in the statements of Propositions~\ref{free-pretopos} and~\ref{free-lext} for $\M=\tx{Flat}(\C,\bo{Set})$, do satisfy an universal property, which we shall describe in Definition~\ref{univ-ultra}.

Given a category $\M$ with products, categorical ultraproducts can be described as filtered colimits of products of elements of $\M$ (see for instance \cite{Lur18:articolo}). More precisely, let $\U$ be an ultrafilter on a set $X$ and $(M_x)_{x\in X}$ be an $X$-indexed family of objects of $\M$, the {\em categorical ultraproduct of $(M_x)_{x\in X}$ over $\U$} (if it exists) is the colimit
\begin{align}\label{ultra}
	\textstyle\prod_\U M_x:=\colim_{S\in \U}\ ( \prod_{s\in S} M_s ). 
\end{align}
Note that this colimit is filtered since whenever $S,T\in\U$ then also $S\cap T\in\U$. 

This means in particular that we have a colimiting cocone
$$ (\delta_S\colon\prod_{s\in S} M_s\longrightarrow \textstyle\prod_\U M_x)_{S\in\U}$$ 
which is universal among all the other cocones over $(\textstyle\prod_{s\in S} M_s)_{S\in\U}$.

Naturally, when $\M$ does not have products the argument above does not work, leading one to define ultraproducts on a category as some additional structure. Below we exhibit a way to express the universal property of ultraproducts which avoids the necessity of dealing with products, and hence can be implemented in any category.

Let us go back to a cocone
$$ (c_S\colon \prod_{s\in S} M_s\longrightarrow C)_{S\in\U}$$ 
over the directed diagram $(\textstyle\prod_{s\in S} M_s)_{S\in\U}$, and let us embed $\M$ into $[\M\op,\bo{Set}]$ through the Yoneda embedding; then the cocone corresponds to a cocone of natural transformations
$$ (\prod_{s\in S}\M(-,M_s) \longrightarrow \M(-,C))_{S\in\U}$$
over $\U$ in $[\M\op,\bo{Set}]$. 

Note that, when $C=\textstyle\prod_\U M_x$ and the cocone is the colimiting cocone $\delta$ in $\M$, then its image $\M(-,\delta)$ in $[\M\op,\bo{Set}]$ does not give a colimiting cocone in $[\M\op,\bo{Set}]$. Nonetheless, $\M(-,\delta)$ is universal among the cocones in $[\M\op,\bo{Set}]$ whose codomain is representable (this is exactly the universal property defining the colimit in $\M$).

This approach motivates the following definition:

\begin{Def}\label{univ-ultra}
	Let $\M$ be a category and $\U$ be an ultrafilter on a set $X$. Given an $X$-indexed family of objects $(M_x)_{x\in X}$, we say that the {\em universal ultraproduct of $(M_x)_{x\in X}$ over $\U$} exists in $\M$ if
	there exists an object $\textstyle\prod_\U M_x\in\M$ together with a cocone
	$$ (\delta_S\colon\prod_{s\in S}\M(-,M_s) \longrightarrow \M(-,\textstyle\prod_\U M_x))_{S\in\U}$$ 
	over $\U$ which is universal among all other cocones with representable codomain. We say that $\M$ has {\em universal ultraproducts} if $\textstyle\prod_\U M_x$ exists for any set $X$, any ultrafilter $\U$ on $X$, and any $X$-indexed family $(M_x)_{x\in X}$ in $\M$.
\end{Def}

It other words, $\delta$ exhibits  $\textstyle\prod_\U M_x\in\M$ as the universal ultraproduct of $(M_x)_{x\in X}$ over $\U$ if and only if for any other cocone $(c_S\colon \textstyle\prod_{s\in S}\M(-,M_s) \longrightarrow \M(-,N))_{S\in\U}$ there exists a unique morphism $f\colon N\to \textstyle\prod_\U M_x$ such that $\M(-,f)\circ c_S=\delta_S$ for any $S\in\U$.

As a consequence of the universal property, if an universal ultraproduct exists in $\M$ then it is unique up isomorphism.

\begin{obs}
	This notion is equivalent to that considered by B\"{o}rger in \cite{Borger82}, by Bui and N\'emeti in \cite{HN81}, and by Rosick\'y in \cite{Ros84models}. Consider a category $\M$, an ultrafilter $\U$ on a set $X$, and a $X$-indexed family of objects $(M_x)_{x\in X}$ in $\M$. Define the category $\Sigma$ with objects triples
	$$ (N,S,u_i):=(N\in\M,S\in\U,(u_s\colon N\to M_s)_{s\in S}) $$%
	and morphisms $(N,S,u_s)\to (N',S',u'_s)$ being determined by 
	$$ (h\colon N\to N', S'\subseteq S,\ hu'_s=u_s). $$%
	This comes with a projection $\pi\colon\Sigma\to\M$. Then \cite{Borger82,HN81,Ros84models} define the universal ultraproduct of $(M_x)_{x\in X}$ over $\U$ as the colimit of $\pi$ in $\M$. It is easy to see that the universal property defining that colimit is the same as that defining our universal ultraproducts. Thus the two notions coincide.
\end{obs}

\begin{obs}
	Let $\M$ have universal ultraproducts; for any ultrafilter $\U$ on a set $X$ we can define a functor
	$$ \textstyle\prod_\U(-)\colon \M^X\longrightarrow\M $$
	given pointwise by sending a family $(M_x)_{x\in X}$ to its universal ultraproduct $\textstyle\prod_\U M_x$. It is easy, but technical, to see that this actually extends to an {\em ultrastructure} on $\M$ in the sense of Lurie~\cite{Lur18:articolo} and of Di Liberti~\cite{di2022geometry}. We also expect this functor to induce an ultrastructure on $\M$ in the original sense defined by Makkai~\cite{Mak87:articolo}, but we leave that for the readers to verify.
\end{obs}

The result below is an immediate consequence of the definition and the arguments above.

\begin{prop}
	If $\M$ has products and filtered colimits then it has universal ultraproducts, and they coincide with those defined using products and filtered colimits (\ref{ultra}).
\end{prop}

The following lemma will be important to recognize when universal ultraproducts exist.

\begin{lema}\label{ultra-close}
	Let $\K$ be a category with universal ultraproducts and let $J\colon\M\hookrightarrow\K$ be a full and dense subcategory of $\K$. If every universal ultraproduct of the form $\textstyle\prod_\U JM_x$ in $\K$, with $M_x\in\M$, lies in $\M$, then $\M$ has universal ultraproducts and $J$ preserves them.
\end{lema}
\begin{proof}
	Let $\U$ be an ultrafilter on a set $X$ and $(M_x)_{x\in X}$ be an $X$-indexed family of objects in $\M$. By hypothesis there is $M\in\M$ such that $JM\cong \textstyle\prod_\U JM_x$; we need to show that $M\cong \textstyle\prod_\U M_x$ in $\M$.
	
	By fully faithfulness of $J$, the colimiting cocone defining $\textstyle\prod_\U JM_x$ in $\K$ gives a unique cocone
	$$ (\delta_S\colon\prod_{s\in S}\M(-,M_s) \longrightarrow \M(-,M))_{S\in\U}$$ 
	over $\U$ in $[\M\op,\bo{Set}]$. We need to prove that it satisfies the universal property defining the universal ultraproduct. 
	
	Note that, since $J$ is dense, then $J\op$ is codense and by the dual of \cite[5.31]{Kel82:libro} the functor
	$$ \tx{Ran}_{J\op}\colon [\M\op,\bo{Set}]\longrightarrow  [\K\op,\bo{Set}]$$
	is fully faithful, continuous, and preserves the representables. Thus every cocone
	$$(c_S\colon \textstyle\prod_{s\in S}\M(-,M_s) \longrightarrow \M(-,N))_{S\in\U}$$
	extend to a cocone 
	$$(d_S\colon \textstyle\prod_{s\in S}\K(-,JM_s) \longrightarrow \K(-,JN))_{S\in\U}$$
	in $\K$ defined by $d_S:=\tx{Ran}_J(c_S)$. Moreover every cocone over $\U$ in $\K$ as above lies in the image of $\tx{Ran}_{J\op}$ (again by \cite[5.31]{Kel82:libro} since the functors involved preserve any existing limits) making it the right Kan extension of its restriction to $\M$.
	\begin{comment}
		Indeed, since $J$ is dense, for each $X\in\K$ there is a diagram $H\colon \D\to\A$ such that $X\cong \colim JH$ and the colimit is $J$-absolute; then 
		$$ \textstyle\prod_{s\in S}\K(X,JM_s)\cong\textstyle\prod_{s\in S}\lim\K(JH-,JM_s)\cong\lim\textstyle\prod_{s\in S}\M(H-,M_s),  $$
		and similarly $\K(X,JN)\cong \lim\M(H-,N)$. Then the cocone $c$ extends to a cocone
		$$(d_S\colon \textstyle\prod_{s\in S}\K(-,JM_s) \longrightarrow \K(-,JN))_{S\in\U}$$
		in $[\K\op,\bo{Set}]$ by taking $(d_S)_X=\lim (c_S)_{H(-)}$; by using density  of $J$ it is easy to check that these are actually natural transformations and that they form a cocone for $\U$. Using the same chain of isomorphisms above, we know also that every cocone for $\U$ in $\K$ of this form is determined by its restriction to $\M$.
	\end{comment}
	Now, given any $c=(c_S)_{S\in\U}$ as above, by the universal property of universal ultraproducts in $\K$ applied to $d=(d_S)_{S\in\U}$, there is a unique map $Jf\colon \textstyle\prod_\U JM_x\to JN$ for which precomposition with $\K(-,Jf)$ maps the colimiting cocone of $\textstyle\prod_\U JM_x$ into the cocone $d$. Restricting along $J$, the map $f$ also satisfies $\M(-,f)\circ c_S=\delta_S$. Such morphism $f$ is unique because whenever $\M(-,f)\circ c_S=\M(-,g)\circ c_S$ then also $\K(-,Jf)\circ d_S=\K(-,Jg)\circ d_S$; thus $Jf=Jg$ by the universal property in $\K$, and hence $f=g$. 
	
	In conclusion $M$ is the universal ultraproduct of the $M_s$ and is preserved by $J$.
\end{proof}

\begin{prop}
	Let $\M$ be a category and $\C$ be a small and dense subcategory of $\M$, so that there is an inclusion $\M\hookrightarrow[\C\op,\bo{Set}]$. The following are equivalent: \begin{enumerate}\setlength\itemsep{0.25em}
		\item $\M$ has universal ultraproducts and they are preserved by $\M(C,-)$ for any $C\in\C$;
		\item $\M$ is closed in $[\C\op,\bo{Set}]$ under categorical ultraproducts.
	\end{enumerate}
\end{prop}
\begin{proof}
	In $\bo{Set}$ universal ultraproducts are the usual categorical ultraproducts; thus the universal ultraproducts of $\M$ are preserved by homming out of $C\in\C$ if and only if they are computed pointwise in $[\C\op,\bo{Set}]$. Then the result follows from Lemma~\ref{ultra-close} above.
\end{proof}

Some direct consequences are:

\begin{cor}
	Let $\tx{Mod}(\C,\bo{Set})$ be the category of models of a pretopos $\C$ such that, when seen as a full subcategory of $[\C,\bo{Set}]$, it contains the representables. Then $\tx{Mod}(\C,\bo{Set})$ has universal ultraproducts and they are preserved by the inclusion.
\end{cor}
\begin{proof}
	It is well known that $\tx{Mod}(\C,\bo{Set})$ satisfies the closure property under ultraproducts in $[\C,\bo{Set}]$. If it also contains the representables then $\tx{Mod}(\C,\bo{Set})$ is dense in $[\C,\bo{Set}]$. Thus it is enough to apply the lemma above. 
\end{proof}

For the following just choose $\C:=\M_f$ in the setting of the proposition above. 

\begin{cor}\label{flat-ultraproducts}
	The following are equivalent for a finitely accessible category $\M$: \begin{enumerate}\setlength\itemsep{0.25em}
		\item $\M$ has universal ultraproducts and they are preserved by $\M(A,-)$ for any $A\in\M_f$;
		\item $\M$ is closed in $[\M_f\op,\bo{Set}]$ under categorical ultraproducts.
	\end{enumerate}
\end{cor}

Note that this extends the equivalent conditions of Proposition~\ref{free-pretopos} by taking $\M=\tx{Flat}(\C,\bo{Set})$.

\begin{obs}
	Given a finitely accessible category $\M$ with universal ultraproducts, we do not know whether these are always preserved as in (1) above.
\end{obs}

\begin{obs}
	In \cite{Ros84models}, a theory $\mathbb T$ over a relational language $\mathbb L$ is called {\em strongly special} if the inclusion $\tx{Mod}(\mathbb T)\hookrightarrow \mathbb L\tx{-Str}$, of the models of $\mathbb T$ into the $\mathbb L$-structures, is fc-reflective (in the sense explained in Remark~\ref{fc-refl}). Then \cite[Proposition 3.3]{Ros84models} shows that $\tx{Mod}(\mathbb T)$ has universal ultraproducts whenever $\mathbb T$ is strongly special. 
	
	Now, given $\M$ as in Corollary~\ref{flat-ultraproducts} above, then $\M$ can be expressed as $\tx{Mod}(\mathbb T)$ for a coherent theory $\mathbb T$ in the language of presheaves on $\M_f\op$ (\cite[4.2]{beke2005flatness}); however, this theory is not strongly special in general (see \cite[4.6]{beke2005flatness} for a counterexample). Thus our corollary does not follow from \cite{Ros84models}. Conversely, \cite[Proposition 3.3]{Ros84models} does not follow from our results since $\tx{Mod}(\mathbb T)$ is not finitely accessible in general.
\end{obs}

\appendix

\section{Finite-cone orthogonality and injectivity}\label{fc-orth-inj}

Categories of models of regular theories (those of the form $\tx{Reg}(\C,\bo{Set})$ for some small regular $\C$) can be characterized as the finite injectivity classes of locally finitely presentable categories \cite{KR18:articolo}. In this section we give a similar characterizations for the categories of models of lextensive categories and of pretopoi by considering notions of finite-cone orthogonality and finite-cone injectivity classes. As a consequence the categories $\tx{Flat}(\C,\bo{Set})$ considered in Sections~\ref{fc} and~\ref{lextesive-sect} can then be captured under these generalized orthogonality and injectivity conditions.

We shall need some fact about categories of fractions \cite[Ch.I]{GZ67:libro}; that is, categories obtained by inverting a class of arrows $\Sigma$ in a (lex) category $\C$. Recall from \cite[Definition~1.1]{Ben89:articolo} that $\Sigma$ is called a {\em pullback congruence} if it contains all the isomorphisms, satisfies the two out of three property, and is stable under pullback. If $\C$ is a regular category, a pullback congruence $\Sigma$ in $\C$ is called {\em regular} if it is local \cite[Definition~2.1.6]{Ben89:articolo}.

When $\C$ is a lextensive category we will need closure under coproducts: a {\em lextensive congruence} on $\C$ is a class $\Sigma$ of maps from $\C$ which is a pullback congruence and is closed under finite coproducts (if $s,t\in\Sigma$ then $s+t\in\Sigma$).

\begin{prop}[\cite{GZ67:libro,Ben89:articolo,aravantinos2017property}]\label{frac}
	Let $\C$ be category with pullbacks and $\Sigma$ a pullback congruence on $\C$, then there exists a category $\C[\Sigma^{-1}]$ together with a functor $$P\colon\C\to \C[\Sigma^{-1}]$$ such that any functor $F\colon\C\to\B$ factors uniquely through $P$ as $F=F_{\Sigma}\circ P$ if and only if $F$ inverts the elements of $\Sigma$. Moreover:
	\begin{enumerate}\setlength\itemsep{0.25em}
		\item If $\C$ is a lex category, then $\C[\Sigma^{-1}]$ is lex, $P$ is a lex functor, and $F_{\Sigma}$ is such if and only if $F$ is.
		\item If $\C$ is regular and $\Sigma$ is a regular congruence, then $\C[\Sigma^{-1}]$ is a regular category, $P$ is a regular functor, and $F_{\Sigma}$ is such if and only if $F$ is.
		\item If $\C$ is lextensive and $\Sigma$ is a lextensive congruence, then $\C[\Sigma^{-1}]$ is a lextensive category, $P$ is a lextensive functor, and $F_{\Sigma}$ is such if and only if $F$ is.
	\end{enumerate}
\end{prop}
\begin{proof}
	The explicit construction of $\C[\Sigma^{-1}]$ can be fond in \cite[Ch.I]{GZ67:libro}. Then $(1)$ and $(2)$ are in \cite{Ben89:articolo}. For $(3)$, we know by $(1)$ that $\C[\Sigma^{-1}]$ is lex and $P$ preserves finite limits. By \cite[Proposition~4.3]{aravantinos2017property} it follows that $\C[\Sigma^{-1}]$ has finite coproducts, that $P$ preserves them, and that $F_{\Sigma}$ preserves finite coproducts if and only if $F$ does. While, that $\C[\Sigma^{-1}]$ is actually extensive is an immediate consequence of \cite[Proposition~2.2]{CLW93:articolo}, the assumption that $\C$ is extensive, and that pullbacks in $\C[\Sigma^{-1}]$ are computed as in $\C$ (up to composition with arrows from $\Sigma$).
\end{proof}

Let us now move to the first result of this section, which aims to characterize the categories of models of small lextensive categories. For that we will need to introduce the notion of orthogonality with respect to finite cones.

\begin{Def}
	Let $\L$ be a locally finitely presentable category and $(c_i\colon C\to C_i)_{i\leq n}$ be a finite cone in $\L$. We say that $M\in\L$ is orthogonal with respect to the cone $(c_i)_i$ if the induced map
	$$ \sum_{i\leq n}\L(C_i,M)\stackrel{\cong}{\longrightarrow}\L(C,M) $$
	is an isomorphism. A {\em finite fc-orthogonality} class in $\L$ is a full subcategory of $\L$ spanned by the objects orthogonal with respect to a (small) set of finite cones with finitely presentable domains and codomains.
\end{Def}

When there is no restrictions on the number of legs of the cone, we recover the {\em (finite) cone orthogonality} classes of \cite[Definition~4.19]{AR94:libro}. 

In the theorem below, we give a characterization of finite fc-orthogonality classes. By a {\em finite-limit/finite-coproduct sketch} we mean a sketch $\S$ whose limit specifications involve only finite diagrams and whose colimit specifications involve finite discrete categories.

\begin{teo}\label{lext-char}
	Let $\M$ be a category; the following are equivalent:\begin{enumerate}\setlength\itemsep{0.25em}
		\item $\M$ is a finite fc-orthogonality class in some $[\C,\bo{Set}]$;
		\item $\M$ is the category of models of a finite-limit/finite-coproduct sketch $\S$;
		\item $\M\simeq\tx{Lext}(\E,\bo{Set})$ for some small lextensive category $\E$.
	\end{enumerate}
	In that case $\M$ is finitely accessible with connected limits (that is, locally finitely multipresentable) and closed under ultraproducts in the ambient category.
\end{teo}
\begin{proof}
	$(3)\Rightarrow (2)$. It is enough to consider the sketch $\S$ on $\E$ consisting of all the limiting cones for finite limits and colimiting cocones for finite coproducts.
	
	$(2)\Rightarrow (1)$. Let $\S=(\C,\mathbb L,\mathbb C)$ be a sketch for which $\M\simeq \tx{Mod}(\S)$. Then the inclusion $\M\hookrightarrow[\C,\bo{Set}]$ can be expressed as the intersection of $\tx{Mod}(\S,\mathbb L)$ and $\tx{Mod}(\S,\mathbb C)$. It is standard that the first can be seen as a finite orthogonality class (since it involves only finite-limit conditions). Regarding the second, given any cocone $(c_i\colon C_i\to C)_{i\leq n}$ in $\mathbb{C}$, a functor $F\colon\C\to \bo{Set}$ sends $(c_i)_i$ to a colimiting cocone if and only if it is orthogonal with respect to the cone $\C(c_i,-)_i$. In conclusion we can express $\M$ as a finite fc-orthogonality class in $[\C,\bo{Set}]$.
	
	$(1)\Rightarrow (3)$. We follow the argument of \cite[Proposition~9.6]{LT20:articolo}. Let $\M$ be a finite fc-orthogonality class in $\K:=[\C,\bo{Set}]$. Note that $\K\simeq \tx{Lex}(\K_f\op,\bo{Set})\simeq \tx{Lext}(\E,\bo{Set})$, where $\E$ is the free lextensive completion of $\K_f\op$ (which exists since $\K_f\op$ is lex). Thus $\M$ can be seen as a finite fc-orthogonality class in $\tx{Lext}(\E,\bo{Set})$ with respect to a class of finite cones $\H$ with representable domains and codomains. Every cone $c\in\H$ corresponds then to a finite cocone $(c_i\colon C_i\to C)_{i\leq n}$ in $\E$ and, since this has finite coproducts, to a map $\bar c\colon \textstyle\sum_i C_i\to C$. It follows easily that $F\colon\E\to\bo{Set}$ is orthogonal with respect to every cone in $\H$ if and only if it sends every $\bar c$ to an isomorphism. \\
	Let now $\Sigma$ be the class of all morphisms $h\in\E$ such that $F(h)$ is an isomorphism for any $F$ in $\M$. By the argument above, a lextensive functor $F\colon\E\to\bo{Set}$ lies in $\M$ if and only if it inverts every morphism in $\Sigma$; moreover, it is easy to see that $\Sigma$ is a lextensive congruence on $\C$. Thus, by Proposition~\ref{frac} above, $\E[\Sigma^{-1}]$ is a lextensive category and satisfies $\M\simeq \tx{Lext}(\E[\Sigma^{-1}],\bo{Set})$.
	
	For the last statement, $\M$ is finitely accessible with connected limits by \cite[Theorem~4.29]{AR94:libro}, while closure under ultraproducts in the ambient category follows from the fact that ultraproducts commute in $\bo{Set}$ with finite coproducts.
\end{proof}

\begin{obs}
	A category satisfying the equivalent conditions of Corollary~\ref{flat-ultraproducts} will also satisfy the three conditions in the theorem above by Proposition~\ref{free-lext}. However, the converse implication need not hold. For instance, the category $\bo{Fld}$ of fields can be expressed as the category of models of a finite-limit/coproduct sketch (being a locally finitely multipresentable category), but does not satisfy the conditions of Corollary~\ref{flat-ultraproducts}. This is because $\bo{Fld}_f$, despite having finite multi-colimits, does not have finite {\em multi-finite} colimits (and hence \ref{free-lext} and \ref{flat-ultraproducts} do not hold): the multi-initial object of $\bo{Fld}_f$ is given by the family consisting of $\mathbb Q$ and $\mathbb Z/p\mathbb Z$ for any prime $p$, which is not finite. 
\end{obs}

\begin{obs}
	As a consequence of the arguments in the remark above, the category $\C:=\bo{Fld}_f\op$ is such that there exists a lextensive category $\E$ together with an equivalence
	$$ \tx{Flat}(\C,\bo{Set})\simeq \tx{Lext}(\E,\bo{Set}), $$
	but $\C$ itself does not have a free lextensive completion.
\end{obs}

\begin{obs}
	Not every finitely accessible category with connected limits can be expressed by a finite-limit/finite-coproduct sketch (or any of the other two equivalent conditions); a counterexample is given in \cite[p.~51]{adamek1996algebraic}. 
\end{obs}

We now move from lextensive categories to pretopoi and try to prove a similar characterization theorem with respect to fc-injectivity conditions.

The following first appeared in \cite[Definition~4.1]{beke2005flatness} as a specialization of the cone injectivity classes studied in \cite{AR94:libro}.

\begin{Def}
	Let $\L$ be a locally finitely presentable category and $(c_i\colon C\to C_i)_{i\leq n}$ be a finite cone in $\L$. We say that $M\in\L$ is injective with respect to the cone $(c_i)_{i\leq n}$ if the induced map
	$$ \sum_{i\leq n}\L(C_i,M)\twoheadrightarrow\L(C,M) $$
	is a surjection. A {\em finite fc-injectivity} class in $\L$ is a full subcategory of $\L$ spanned by the objects injective with respect to a (small) set of finite cones with finitely presentable domains and codomains.
\end{Def}

\begin{obs}\label{orth-inj}
	Every finite fc-orthogonality class is also a finite fc-injectivity class by \cite[Proposition~4.23]{AR94:libro}. Moreover, each finite fc-injectivity class is closed under filtered colimits and ultraproducts. Indeed, the first property holds more generally for finite cone injectivity classes, while the latter can be obtained by an easy computation using that the cones involved are finite and that ultrafilters have the finite intersection property.
\end{obs}

Recall from \cite{KR18:articolo} that a full subcategory $\D$ of a locally finitely presentable category is a finite injectivity class if and only if there exists a regular category $\C$ such that $\D\simeq\tx{Reg}(\C,\bo{Set})$. We show below that a similar result holds for categories of models of pretopoi.

In point $(2)$ below we consider a generalization of the notion of limit/epi sketch of \cite[Definition~4.12]{AR94:libro}. We say that a sketch $\S=(\C,\mathbb L,\mathbb C)$ is a {\em finite-limit/fc-epi} sketch if the limit specifications only involve finite diagrams, and every colimit diagram is given by a finite family of spans as below
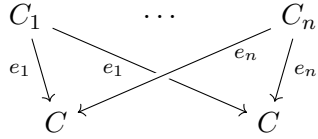
\begin{center}
	\begin{tikzpicture}[, baseline=(current  bounding  box.south), scale=2, on top/.style={preaction={draw=white,-,line width=#1}}, on top/.default=4pt]
		
		\node (a'0) at (0,0.7) {$C_1$};
		\node (0) at (0.9,0.7) {$\dots$};
		\node (a'''0) at (1.8,0.7) {$C_n$};
		\node (b0) at (1.6,0) {$C$};
		\node (c0) at (0.2,0) {$C$};
		
		\path[font=\scriptsize]

		(a'0) edge [->] node [left] {$e_1\ \ $} (b0)
		(a'0) edge [->] node [left] {$e_1$} (c0)

		(a'''0) edge [->] node [right] {$e_n$} (b0)
		(a'''0) edge [->, on top] node [above] {$\ \ \ \ \ \ \ \ \ \ \ \ \ \ \ \ \ \ \ e_n$} (c0);
	\end{tikzpicture}	
\end{center}
with cocone specifications given by the identity cospan on $C$. Thus, a functor $F\colon\C\to\bo{Set}$ is a model of $\S$ if and only if it sends the specified cones to limits, and the specified cocone $(e_i\colon C_i\to C)_{i\leq n}$ in $\mathbb C$ to a jointly epimorphic family.

Then the next theorem characterizes finite fc-injectivity classes. From the point of view of logic, considering that pretopoi axiomatize coherent theories, this can also be seen as consequence of the results of \cite{AR94:libro} (see Remark~\ref{conj-disj} below).

\begin{teo}
	Let $\M$ be a category; the following are equivalent:\begin{enumerate}\setlength\itemsep{0.25em}
		\item $\M$ is a finite fc-injectivity class in some $[\C,\bo{Set}]$;
		\item $\M$ is the category of models of a finite-limit/fc-epi sketch;
		\item $\M$ is the category of models of a finite-limit/(finite-coproduct,epi) sketch;
		\item $\M\simeq\tx{Mod}(\E,\bo{Set})$ for some small pretopos $\E$.
	\end{enumerate} 
	In that case $\M$ is accessible and closed under ultraproducts in the ambient presheaf category. 
\end{teo}
\begin{proof}
	$(4)\Rightarrow(3)$. This is trivial since a functor $F\colon\E\to\bo{Set}$ is a model of $\E$ if and only if it preserves finite limits, finite coproducts, and regular epimorphisms. 
	
	$(3)\Rightarrow(2)$. It is enough to express the finite-coproduct specifications using finite limits and jointly epimorphic families. And this is easy since a functor sends a cocone  $(c_i\colon C_i\to C)_{i\leq n}$ to a coproduct cocone if and only if it sends them to a family that is both jointly monomorphic and epimorphic.
	
	$(2)\Rightarrow (1)$. The finite-limit specifications can be expressed as orthogonality conditions, which by Remark~\ref{orth-inj}, correspond to some finite fc-injectivity conditions. Finally, to say that $F$ sends $(c_i\colon C_i\to C)_{i\leq n}$ to a jointly epimorphism family is the same as saying that $F$ is injective with respect to the finite cone $\C(c_i,-)_{i\leq n}$.
	
	$(1)\Rightarrow (4)$. We argue as in the proof of Theorem~\ref{lext-char}. First we see $\M$ as a finite fc-injectivity class in $\tx{Mod}(\E,\bo{Set})$ with respect to a class of finite cones $\H$ with representable domains and codomains (here $\E$ is the free pretopos completion of $[\C,\bo{Set}]_f$). Every cone $c\in\H$ corresponds to a finite cocone $(c_i\colon C_i\to C)_{i\leq n}$ in $\E$ and, since this has finite coproducts, to a map $\bar c\colon \textstyle\sum_i C_i\to C$. Let $\bar c= e\circ m$ be the (regular epi,mono) factorization of $\bar c$; then a functor $F\in \tx{Mod}(\E,\bo{Set})$ is injective with respect to $c$ if and only if it sends $m$ to an isomorphism. \\
	Thus, if we consider the class $\Sigma$ of all morphisms $h\in\E$ such that $F(h)$ is an isomorphism for any $F$ in $\M$, by the argument above we obtain that $F\in \tx{Mod}(\E,\bo{Set})$ lies in $\M$ if and only if it inverts every morphism in $\Sigma$. Since every $F\in\M$ preserves coproducts and regular epimorphisms moreover, it is easy to see that $\Sigma$ is both a regular and lextensive congruence on $\C$. Thus, by Proposition~\ref{frac} above, $\E[\Sigma^{-1}]$ is a regular and lextensive category and we have an equivalence $\M\simeq \tx{Mod}(\E[\Sigma^{-1}],\bo{Set})$ (as for pretopoi, a model $\E[\Sigma^{-1}]\to\bo{Set}$ is a lex functor preserving finite coproducts and regular epimorphisms). To conclude it is then enough to consider the relative free pretopos completion of $\E[\Sigma^{-1}]$ as a regular-lextensive category; this is given by \cite[Theorem~7.7]{GL12:articolo}.
	
	For the last statement, $\M$ is accessible by \cite[Theorem~4.17]{AR94:libro}, while closure under ultraproducts in the ambient category follows from the fact that finite jointly epimorphic families are stable under ultraproducts in $\bo{Set}$.
\end{proof}

When $\M$ is finitely accessible this can be seen as a consequence of \cite[3.3  and~4.2]{beke2005flatness}. Every category $\M$ as in the theorem above is accessible; however, unlike in Theorem~\ref{lext-char}, the category need not be {\em finitely} accessible. The same problem with cardinality holds also for finite injectivity classes in the theory of definable categories \cite{Pre11:libro,KR18:articolo}.

\begin{obs}\label{conj-disj}
	From the point of view of model theory, the theorem above can be seen as a consequence of the results in \cite[Section~5.B]{AR94:libro}. Indeed, it is essentially saying that finite fc-injectivity classes in presheaf categories correspond to categories of models of coherent logic; meaning those whose axioms are of the form
	$$ \forall x (\phi(x)\Rightarrow\psi(x)) $$%
	where $\phi$ and $\psi$ are finite disjunctions of positive-primitive formulas. 
\end{obs}

%\bibliography{biblio}
%\bibliographystyle{abbrv}

\end{document}